\documentclass{article}
\usepackage{graphicx}
\usepackage{xcolor}
\usepackage{kpfonts}
\usepackage{mathptmx} %cambia tipo di scrittura
\usepackage{amsmath} \usepackage{amsfonts}%simboli matematici
\usepackage[margin=1.5in]{geometry}
\usepackage{amsthm} %teoremi definizioni ecc
\usepackage{enumitem} %numerazione latina
\theoremstyle{definition}
\newtheorem{definition}{Definition}

\newtheorem{theorem}{Theorem}
\newtheorem{lemma}{Lemma}
\newtheorem{remark}{Remark}
\newtheorem{corollary}{Corollary}
\numberwithin{equation}{section}
\numberwithin{theorem}{section}
\numberwithin{proposition}{section}
\numberwithin{lemma}{section}
\numberwithin{corollary}{section}
\numberwithin{remark}{section}
\usepackage{hyperref}

\providecommand{\msc}[1]{\small	\textbf{2020 Mathematics Subject Classification:} #1}
\providecommand{\keywords}[1]
{\small	\textbf{\textit{Keywords:}} #1}
\title{Weak convergence of the integral of semi-Markov processes}
\author{Andrea Pedicone$^1$ and Fabrizio Cinque$^2$\\\small Department of Statistical Sciences, Sapienza University of Rome, Italy \\\small $^1$andrea.pedicone@uniroma1.it $^2$fabrizio.cinque@uniroma1.it}
\date{}
\begin{document}
\maketitle
\begin{abstract}
We study the asymptotic properties, in the weak sense, of regenerative processes and Markov renewal processes. For the latter, we derive both renewal-type results, also concerning the related counting process, and ergodic-type ones, including the so-called phi-mixing property. This theoretical framework permits us to study the weak limit of the integral of a semi-Markov process, which can be interpret as the position of a particle moving with finite velocities taken for a random time according to the Markov renewal process underlying the semi-Markov one. Under mild conditions, we obtain the weak convergence to scaled Brownian motion. As a particular case, this result establishes the weak convergence of the classical generalized telegraph process.
\end{abstract}
\keywords{Weak convergence; Semi-Markov processes; Markov renewal processes; Regenerative processes; Mixing sequences; Ergodic Theory}
\\
\msc{Primary 60F17, 60K15; Secondary 60F05}
% 60F05 Central limit and other weak theorems
% 60F17 Functional limit theorems; invariance principles
% 60K15 Markov renewal processes, semi-Markov processes
\section{Introduction}

Let $V = (V(t))_{t\geq 0}$ be a cadlag stochastic process on a state space $\mathcal{V}$. Let us define $X = (X(t))_{t \geq 0}$, the integral of $V$, as 
\begin{equation}\label{telegraph process}
X(t) = \int_0^tV(s)ds,\ \;\; t \geq 0.
\end{equation}
The process \eqref{telegraph process}, under different formulation of $V$, includes a wide class of stochastic motions which appear in the probabilistic and physical literature under several different names, such as telegraph-type processes, finite-velocity random motions \cite{BNO2001, CC2024, Dg2010, Dc2001, O1990, SZ2004}, continuous time or directionally reinforced random walks \cite{HS1998, MMW1996, MS2014} and run-and-tumble processes \cite{DbMS2021, MADlB2012}. Our interest concerns the asymptotic behavior of $X$ in the case where $V$ is semi-Markov process. Under this probabilistic structure we formalize a model for the motion of a particle moving with a semi-Markovian velocity process, meaning that it keeps a certain speed for a random time dependent on both the current velocity and the following one. Our formulation is broad enough to include all the cited processes as particular cases. 

Let $(\hat{V},S)$ the Markov renewal process related to $V$, equation \eqref{telegraph process}, can be expressed in the following equivalent form
\begin{equation}\label{telegraph process 2}
X(t) = \sum_{k=1}^{N(t)}\hat{V}_{k-1}(S_k-S_{k-1})+\hat{V}_{N(t)}(t-S_{N(t)})
\end{equation}
where $N$ is the counting process associated to $S$, \textit{i.e.} $N(t) = \max\{k\in \mathbb{N}_0:S_k\leq t\}$. 

The main result of this paper is the weak convergence of a suitable normalization of $X(\lambda t)$ to a scaled Brownian motion in the space of the continuous function $C[0,+\infty)$ endowed with the topology of the uniform metric. By defining the sequence of stochastic process $X_\lambda = (X_\lambda(t))_{t\geq 0}$ as $X_\lambda(t) = \lambda^{-1/2}(X(\lambda t)-\lambda \theta t)$, where $\theta = (\mathsf{E}_\pi[S_1])^{-1}\mathsf{E}_\pi[\hat{V}_0S_1]$ (with $\mathsf{E}_\pi$ stands for the expected value computed under the invariant measure $\pi$ of $\hat{V}$), we prove that
\begin{equation}
X_\lambda \Rightarrow \mu^{-1/2}\gamma W
\end{equation}
where $W$ is Brownian motion, $\mu = \mathsf{E}_\pi[S_1]$, and
\begin{align}
\gamma^2 &= \mathsf{E}_\pi[(\hat{V}_0-\theta)^2S_1^2]+2\sum_{k\geq 1}\mathsf{E}_\pi[(\hat{V}_0-\theta)(\hat{V}_k-\theta)S_1(S_{k+1}-S_k)].
\end{align}
Moreover, we will also show that the above formula can be expressed as
\begin{equation}
\gamma^2=\pi_{v_0}\mathsf{E}\biggl[\biggl(\sum_{k=1}^{\tau_1}(\hat{V}_{k-1}-\theta)S_1\biggr)^{2}\bigg|\hat{V}_0 = v_0\biggr].
\end{equation}
where $\tau_1$ is the first passage times over the state $v_0$ of the Markov chain $\hat{V}$, that is $\tau_1 = \inf\{k>1: \hat{V}_k = v_0\}$.

In order to fulfill our purpose %, we take advantage of both a regenerative property and a mixing property, both satisfied by the sequence $\{(\hat{V}_{k-1},S_k-S_{k-1})\}_{k\geq 1}$. To be specific, 
we prove that the sequence $\{f(\hat{V}_{k-1},S_k-S_{k-1})\}_{k \geq 1}$, where $f$ is any measurable function, is $\varphi$-mixing with $\varphi_k = K\rho^{k-1}$, for some $K>0$, $\rho \in [0,1)$, and delayed regenerative with respect to the sequence of successive passage times through a fixed state of the Markov chain $\hat{V}$. See \cite{Sigman Wolff} for a review of regenerative processes and references therein for their applications. 
\\The regenerative property permits us to derive a weak version of a renewal-type theorem concerning the related counting process, that is, for some $T>0$
\begin{equation}
\lim_{n \to +\infty}\mathsf{P}\biggl\{\sup_{t \in [0,T]}\biggl|\frac{N(n t)}{n}-\frac{t\mathsf{E}[\tau_1|\hat{V}_0 = v_0]}{\mathsf{E}[\sum_{k=1}^{\tau_1}(S_k-S_{k-1})|\hat{V}_0 = v_0]}\biggr| \geq \epsilon \biggr\} = 0,\ \;\; \epsilon >0.
\end{equation}
Then, we prove an ergodic theorem for the sequence $\{f(\hat{V}_{k-1},S_k-S_{k-1})\}_{k \geq 1}$
\begin{equation}
\lim_{n\to +\infty}\frac{1}{n}\sum_{k=1}^nf(\hat{V}_{k-1},S_k-S_{k-1}) = \mathsf{E}_\pi[f(\hat{V}_0,S_1)]
\;\;a.s.
\end{equation}
where $f$ is any measurable function such that $\mathsf{E}_\pi[|f(\hat{V}_0,S_1)|]<+\infty$.

Our problem has been inspired by the theory of the telegraph process (see \cite{O1990}), which arises from \eqref{telegraph process} by setting $V(t) = V(0)(-1)^{N(t)}$, where $N = (N(t))_{t\geq 0}$ is a Poisson process of parameter $\lambda>0$ and $V(0)$ is uniform over $\{-c,c\}$, with $c>0$, and independent of $N$. We point out that the main focus of the research in this area is in finding the explicit distributions of telegraph-type processes (see \cite{BNO2001, C2022, Dg2010, LR2014}), their functionals and conditioned process \cite{C2023, CO2021, LR2014, O1990, PO2025}. By considering the equation that governs the absolutely continuous component of the law of the telegraph process
\begin{equation}\label{telegrapher equation}
\frac{\partial^2 u}{\partial t^2}+2\lambda \frac{\partial u}{\partial t} = c^2\frac{\partial^2u}{\partial x^2},
\end{equation}
which is known as telegraph equation, we see that that, for $\lambda,c\to+\infty$ such that $c^2/\lambda\to1$ (the so called Kac's limit conditions), \eqref{telegrapher equation} becomes the heat equation. Under this behavior the particle speed goes to infinity as well as the number of inversions, maintaining a specific ratio between the two limits. This consideration is the heuristic explanation for the result proved in \cite{EK1986} (page 471), where it is shown that the telegraph process converges weakly to Brownian motion in the space $C[0,+\infty)$. We also refer to \cite{O1990, LR2014} for the pointwise convergence, to \cite{C2022} for a conditional one and to \cite{PO2025} for the convergence of the telegraph meander. 
\\The weak convergence of telegraph-type processes has been studied in \cite{Gosh, HS1998}, where the authors proved an invariant principle and related limit theorems in the case where $V(t) = \hat{V}_{k-1}\mathsf{1}_{[S_{k-1},S_k)}(t)$, under the stronger hypothesis that $\{\hat{V}_k\}_{k\geq 0}$ is an irreducible stationary Markov chain with finite $d$-dimensional state space, independent of $\{S_k\}_{k\geq 0}$, which is a a renewal sequence. 

This paper is organized in the following manner. Section \ref{SezioneProprietaAsintoticheProcessiRigenerativi} contains some preliminary limit result about sequences indexed by integers random variables and some asymptotic results concerning the regenerative processes. Section \ref{SezioneProprietaAsintoticheMarkovRinnovo} provides a detailed study of the limit behavior of Markov renewal processes. Section \ref{sezioneConvergenzaDeboleIntegraleSemiMarkov} concerns the weak limit of the integral \eqref{telegraph process} under a semi-Markov process. Finally, in Section \ref{sezioneConvergenzaIntegralRinnovoAlternato} we study the limit of the integral \eqref{telegraph process} under an alternating renewal process, a particular case of a semi-Markov process. As an application, we establish the weak convergence of the classical generalized telegraph process. 
\section{Limit theorems for regenerative processes}\label{SezioneProprietaAsintoticheProcessiRigenerativi}
For a given a stochastic process $\eta = (\eta(t))_{t\geq0}$, suppose that there exists a random time $\tau$ such that the process $\big(\eta(t+\tau);t\geq 0\big)$ has the same distribution as $\eta$ and is independent of the past cycle $\big\{\big(\eta(t);0\leq t<\tau\big),\tau\big\}$. Then we say that time $\tau$ is a regeneration epoch and $\eta$ regenerates at time $\tau$. This means that the process restart as if it was time $t = 0$ again, and its future is independent of its past. But if such a $\tau$ exists, then, since things start over again as if new, there must be a second time $\tau' > \tau$ yielding an identically distributed second cycle $\big\{(\eta(\tau+t);0\leq t<\tau'-\tau\big),\tau'-\tau\big\}$ and so on. This probabilistic structure describes what is called  regenerative process. Below we give the formal definition.   
\begin{definition}\label{Regenerative 2}
A %cadlag 
stochastic process $\eta = (\eta(t))_{t\geq 0}$, defined on a filtered probability space $(\Omega,\mathcal{G},(\mathcal{F}_t)_{t\geq 0},\mathsf{P})$, is said to be regenerative if there exists a sequence $\{\tau_m\}_{m\geq 0}$, $\tau_0 = 0$, of stopping times with respect to $(\mathcal{F}_t)_{t\geq 0}$ provided that $\{\tau_{m+1} - \tau_{m},(\eta(t+\tau_{m});\;\;t \in [0,\tau_{m+1} - \tau_{m}))\}_{m\geq 0}$ forms a sequence of i.i.d. random elements.%\textcolor{red}{Sequence of sigma algebras?} 
\\The process $\eta$ is delayed regenerative, if $
\{\tau_{m+1} - \tau_{m},(\eta(t+\tau_{m});\;\;t \in [0,\tau_{m+1} - \tau_{m}))\}_{m\geq 1}$ is a sequence of i.i.d. random elements independent of $(\tau_{1} ,\eta(t)_{t \in[0,\tau_{1})})$.
\end{definition}
%The random times that form the renewal sequence $\{\tau_m\}_{m\geq 1}$ are called regeneration epochs. 
According to \cite{Cinlar} page 298, another way to define a regenerative process is the following.
\begin{definition}\label{Regenerative}
A %cadlag
stochastic process $\eta = (\eta(t))_{t\geq 0}$, defined on a filtered probability space $(\Omega,\mathcal{G},(\mathcal{F}_t)_{t\geq 0},\mathsf{P})$ with state space $E$, is said to be regenerative provided that 
\begin{enumerate}[label=(\roman*)]
\item there exists a sequence $\{\tau_m\}_{m\geq 0}$, $\tau_0 = 0$, of stopping times with respect to $(\mathcal{F}_t)_{t\geq 0}$ that forms a renewal sequence
\item for any $d,m \in \mathbb{N}$, $0\leq t_1<...<t_d$ and any measurable and bounded function $g:E^d \to \mathbb{R}$
\begin{equation*}
\mathsf{E}[g(\eta(t_1+\tau_m),...,\eta(t_d+\tau_m))|\mathcal{F}_{\tau_m}] = \mathsf{E}[g(\eta(t_1),...,\eta(t_d))].
\end{equation*}
\end{enumerate}
The process $\eta$ is delayed regenerative, if 
\begin{enumerate}[label={(\roman*a)}]
\item there exists a sequence $\{\tau_m\}_{m\geq 0}$, $\tau_0 = 0$, of stopping times with respect to $(\mathcal{F}_t)_{t\geq 0}$ that forms a delayed renewal sequence
\item for any $d,m \in \mathbb{N}$, $0\leq t_1<...<t_d$ and any measurable and bounded function $g:E^d \to \mathbb{R}$
\begin{equation*}
\mathsf{E}[g(\eta(t_1+\tau_m),...,\eta(t_d+\tau_m))|\mathcal{F}_{\tau_m}] = \mathsf{E}[g(\eta(t_1+\tau_1),...,\eta(t_d+\tau_1))].
\end{equation*}
\end{enumerate}
\end{definition}
Loosely speaking, a stochastic process is regenerative if there is a renewal process such that the segments of the process between successive renewal times are i.i.d.. While, a stochastic process is delayed regenerative if the first cycle has a distribution that is different from that of subsequent cycles.

The i.i.d. structure underlying a regenerative process allows us to establish classical limit theorems for sequence of random variables. Before doing so, we state the following two lemmas (which we will frequently use in our work) about convergence of randomly indexed sequences of random variables. As it will be shown, a random sequence indexed by a random variable has the same limit as the corresponding sequence indexed by a deterministic sequence, provided that the random index is "asymptotically equivalent" to the deterministic index. The first lemma shows an invariance principle in $D[0,1]$ (the space of real càdlàg function endowed with the Skorohod topology) of a random walk with random number of summands. The proof follows the argument of Theorem 14.4 in \cite{Billingsley} and will be omitted.
%\begin{lemma}Let $\{\eta_k\}_{k\geq 1}$ be a sequence of i.i.d. random variables with $\mathsf{E}[\eta_1] = 0$, $\mathsf{E}[\eta_1^2] = \sigma^2 <+\infty$ and $(\nu_t)_{t\geq 0}$ be a family of integer random variables. Suppose that there exist a divergent family of positive constant $a_t$ and $\theta$ such that\begin{equation}\label{nu asy equal to a 1}\frac{\nu_t}{a_t} \Rightarrow \theta.\end{equation}Then,\begin{equation}\label{S_nu - S_a converges in probability to zero}%\lim_{t\to +\infty}\mathsf{P}\{|S_{\nu_t}-S_{\lfloor a_t\rfloor}|\geq \epsilon \sqrt{a_t}\} = 0\frac{\sum_{k=1}^{\nu_t}\eta_k- \sum_{k=1}^{\lfloor \theta a_t\rfloor}\eta_k}{\sqrt{a_t}} \Rightarrow 0\end{equation}and hence,\begin{equation}\label{S_nu asy normal}\frac{\sum_{k=1}^{\nu_t}\eta_k}{\sqrt{a_t\theta}\sigma} \Rightarrow N(0,1).\end{equation}\end{lemma}\begin{proof}First, use \eqref{nu asy equal to a 1} and apply the Kolmogorov's inequality to get \eqref{S_nu - S_a converges in probability to zero}. Then, use the central limit theorem to deduce \eqref{S_nu asy normal}.\end{proof}We state a slight modification of Theorem 14.4 of \cite{Billingsley}, which generalizes the previous lemma in a functional setting. 
\begin{lemma}\label{Fundamental lemma}
Let $\{\eta_k\}_{k\geq 1}$ be a sequence of random variables with $\mathsf{E}[\eta_k] = 0$, $\mathsf{E}[\eta_k^2] = \sigma^2 <+\infty$ for every $k \in \mathbb{N}$. Define the processes
\begin{align}
&X_n(t) = \frac{1}{\sqrt{n}\sigma}\sum_{k=1}^{\lfloor nt\rfloor}\eta_k,\ \;\; t \in [0,1]\\
&Y_n(t) = \frac{1}{\sqrt{a_n\theta} \sigma}\sum_{k=1}^{\nu_{nt}}\eta_k,\ \;\; t \in [0,1]
\end{align}
where $\theta$ is a positive constant, $\{a_n\}$ is a positive divergent sequence and $(\nu_{t})_{t \geq 0}$ is a family of integer random variables. If it holds that
\begin{equation}
\sup_{t \in [0,1]}\biggl|\frac{\nu_{nt}}{a_n}-\theta t\biggr| \Rightarrow 0,
\end{equation}
then $X_n \Rightarrow W$ on $D[0,1]$, implies $Y_n \Rightarrow W$ on $D[0,1]$.
\end{lemma}
\begin{lemma}\label{lemma 3}
Let $\{\eta_k\}_{k\geq 1}$ be a sequence of identically distributed random variables with $\mathsf{E}[|\eta_1|^p]<+\infty$ for some $p>0$. Let $(\nu_t)_{t\geq 0}$ be a family of integer random variables and suppose that there exist a positive divergent family $(a_t)_{t\geq0}$ and $\theta>0$ such that
\begin{equation}\label{nu asy equal to a 2}
\frac{\nu_t}{a_t} \Rightarrow \theta.
\end{equation}
Then,
\begin{equation}
\max_{k=1,...,\nu_t}\frac{\eta_k}{\sqrt[p]{a_t}} \Rightarrow 0.
\end{equation}
\end{lemma}
\begin{proof}
For every $\epsilon>0$ and $\delta>0$ we have that
\begin{align*}
\mathsf{P}\{\max_{k = 1,...,\nu_t}|\eta_k|>\epsilon \sqrt[p]{a_t}\} &\leq \mathsf{P}\{\max_{k = 1,...,\nu_t}|\eta_k|>\epsilon\sqrt[p]{a_t},|\nu_{t}-\theta^{-1}a_t|<\delta a_t\}+\mathsf{P}\{|\nu_{t}-\theta^{-1}a_t|\geq\delta a_t\} \\
&\leq \mathsf{P}\{\max_{k = 1,...,\lfloor a_t(\theta^{-1} +\delta) \rfloor}|\eta_k|>\epsilon\sqrt[p]{a_t}\}+\mathsf{P}\{|\nu_{t}-\theta^{-1}a_t|\geq\delta a_t\}\\
&\leq\frac{\lfloor a_t(\theta^{-1} +\delta) \rfloor}{\epsilon^p a_t}\mathsf{E}[|\eta_1|^p\mathsf{1}_{\{|\eta_1|\geq \epsilon\sqrt[p]{a_t}\}}]+\mathsf{P}\{|\nu_{t}-\theta^{-1}a_t|\geq\delta a_t\}.
\end{align*}   
Now, by applying the dominate convergence theorem, the first term of the previous inequality goes to zero, while the second one goes to zero by means of \eqref{nu asy equal to a 2}.
\end{proof}
The next statements concern the asymptotic behavior of regenerative sequences, showing a functional central limit theorem and a strong law of large numbers.
\begin{theorem}\label{Clt regenerative processes}
Let $\{\eta_k\}_{k\geq 1}$ be a delayed regenerative process with regeneration epochs $\{\tau_m\}_{m\geq 0}$, $\tau_0 = 0$, such that $\mathsf{E}[\tau_m^2] < +\infty$ for every $m \in \mathbb{N}$. If we define the random function
\begin{equation}
X_n(t) = \frac{1}{\sqrt{n}\sigma}\sum_{k=1}^{\lfloor nt\rfloor}(\eta_k-\mu),\ \;\; t \in [0,1]
\end{equation}
where
\begin{align}\label{mu sigma reg}
&\mu = \frac{\mathsf{E}[\sum_{k=\tau_1+1}^{\tau_2}\eta_k]}{\mathsf{E}[\tau_2-\tau_1]}, &\sigma^2 = \frac{\mathsf{E}[\bigl(\sum_{k=\tau_1+1}^{\tau_2}(\eta_k-\mu)\bigr)^2]}{\mathsf{E}[\tau_2-\tau_1]},
\end{align}
and, if we assume that for every $m \in \mathbb{N}$
\begin{equation}\label{assumption finite variance}
\mathsf{E}\biggl[\biggl(\sum_{k=\tau_{m-1}+1}^{\tau_{m}}|\eta_k-\mu|\biggr)^2\biggr] <+\infty,
\end{equation}
then, it holds that
\begin{equation}
X_n \Rightarrow W
\end{equation}
in the space $D[0,1]$.
\end{theorem}
\begin{proof}
Let us define 
\begin{equation*}
R_m = \sum_{k=\tau_{m-1}+1}^{\tau_{m}}\frac{(\eta_k-\mu)}{\sigma},\ \;\; m \in \mathbb{N}.
\end{equation*}
By the regenerative property, $\{R_m\}_{m \geq 1}$ defines a sequence of independent random variables that are identically distributed for $m\geq 2$. We have that $\mathsf{E}[R^2_1] <+\infty$, $\mathsf{E}[R_m] = 0$ and $\mathsf{E}[R_m^2] =\mathsf{E}[\tau_2-\tau_1]< +\infty$, for every $m \geq 2$, by \eqref{assumption finite variance}. Let us also denote with $\nu_n = \max\{m \in \mathbb{N}_0: \tau_m \leq n\}$. Then, the process
\begin{equation*}
Y_n(t) = \frac{1}{\sqrt{n}}\sum_{m=2}^{\nu_{\lfloor n t\rfloor}}R_m,\ \;\; t \in [0,1]
\end{equation*}
is a random walk with i.i.d. increments and a random number of summands. It follows that the process $X_n(t)$ can be reformulated as
\begin{equation*}
X_n(t) = \frac{R_1}{\sqrt{n}}+Y_n(t) + \frac{1}{\sqrt{n}\sigma}\sum_{k=\tau_{\nu_{\lfloor n t\rfloor}}+1}^{\lfloor n t\rfloor}(\eta_k -\mu).
\end{equation*}
%By the strong law of large numbers applied to the renewal process $\{\nu_n\}$\begin{equation}\biggl|\frac{\eta_{nt}}{n}-\frac{t}{\theta}\biggr| \Rightarrow 0\end{equation}where $\theta = \mathsf{E}[\tau_1]$
Now, the inter-arrivals time $\zeta_m = \tau_m-\tau_{m-1}$ are independent for every $m \in \mathbb{N}$ and identically distributed for $m\geq 2$, with $\mathsf{E}[\zeta^2_m] <+\infty$ for every $m \in \mathbb{N}$ by assumption. Then, consider the random walk built on $\{\zeta_k\}_{k\geq1}$,
\begin{equation*}Z_n(t) = \frac{1}{\sqrt{n\mathsf{V}[\tau_1]}}(\zeta_1-\mathsf{E}[\tau_1])+\frac{1}{\sqrt{n\mathsf{V}[\tau_2-\tau_1]}}\sum_{k=2}^{\lfloor nt\rfloor}(\zeta_k-\mathsf{E}[\tau_2-\tau_1]),\ \;\; t \in [0,1].
\end{equation*}
By applying the Borel-Cantelli lemma, the first term converges almost surely to zero, and hence, by Donsker's Theorem (Theorem 14.1 \cite{Billingsley}), it holds $Z_n \Rightarrow W$ in $D[0,1]$. From this, with Theorem 14.6 of \cite{Billingsley} at hand, we have that the process $\bar{\nu}_n$
\begin{equation*}
\bar{\nu}_n(t) = \frac{\nu_{\lfloor n t\rfloor}-\mathsf{E}[\tau_2-\tau_1]^{-1}nt}{\mathsf{E}[\tau_2-\tau_1]^{-\frac{3}{2}}\mathsf{V}[\tau_2-\tau_1]\sqrt{n}},\ \;\; t \in [0,1]
\end{equation*}
is such that $\bar{\nu}_n \Rightarrow W$ in $D[0,1]$. Moreover, by the continuous mapping theorem, we can obtain a uniform basic renewal theorem
\begin{equation}\label{flln rp}\sup_{t \in [0,1]}\biggl|\frac{\nu_{\lfloor n t\rfloor}}{n}-\frac{t}{\mathsf{E}[\tau_2-\tau_1]}\biggr| \Rightarrow 0.
\end{equation}
Hence, we use once again Donsker's Theorem to prove that
\begin{equation}
\hat{Y}_n(t) = \frac{1}{\sqrt{n}}\sum_{m=2}^{\lfloor n t\rfloor}R_m,\ \;\; t \in [0,1],
\end{equation}
satisfies $\hat{Y}_n\Rightarrow W$, and then by applying Lemma \ref{Fundamental lemma}, it follows $Y_n \Rightarrow W$ in $D[0,1]$.

It remains to prove that
\begin{equation*}\label{resid to zero}
\frac{R_1}{\sqrt{n}}+\sup_{t \in [0,1]}\biggl|\frac{1}{\sqrt{n}\sigma}\sum_{k=\tau_{\nu_{\lfloor n t\rfloor}}+1}^{\lfloor n t\rfloor}(\eta_k -\mu)\biggr| \Rightarrow 0.
\end{equation*}
As before, the first term converges almost surely to zero by the Borel-Cantelli lemma, hence we focus on the second one. We have the following inequalities
\begin{align*}\label{resid goes to zero}
\sup_{t \in [0,1]}\biggl|\frac{1}{\sqrt{n}\sigma}\sum_{k=\tau_{\nu_{\lfloor n t\rfloor}}+1}^{\lfloor n t\rfloor}(\eta_k -\mu)\biggr| &\leq \sup_{t \in [0,1]}\biggl(\frac{1}{\sqrt{n}\sigma}\sum_{k=\tau_{\nu_{\lfloor n t\rfloor}}+1}^{\tau_{\nu_{\lfloor n t\rfloor+1}}}|\eta_k -\mu|\biggr) \nonumber\\&= \max_{m = 0,...,\nu_n}\biggl(\frac{1}{\sqrt{n}\sigma}\sum_{k=\tau_{m}+1}^{\tau_{m+1}}|\eta_k -\mu|\biggr).
\end{align*}
Assumption \eqref{assumption finite variance} and result \eqref{flln rp} permit us to apply Lemma \ref{lemma 3} from which we obtain that the last term in the previous inequality goes to zero. The proof is now concluded.
%and then we have, for every $\epsilon>0$ and $\delta>0$\begin{align*}&\mathsf{P}\{\sup_{t \in [0,1]}\tilde{R}_{\nu_{\lfloor n t\rfloor}}>\epsilon \sqrt{n}\}=\mathsf{P}\{\max_{k = 1,...,\nu_n}\tilde{R}_k>\epsilon \sqrt{n}\} \\&\leq \mathsf{P}\{\max_{k = 1,...,\nu_n}\tilde{R}_k>\epsilon\sqrt{n},|\nu_{nt}-\theta^{-1}nt|<\delta n\}+\mathsf{P}\{|\nu_{nt}-\theta^{-1}nt|\geq\delta n\} \\&\leq \mathsf{P}\{\max_{k = 1,...,\lfloor n(\theta^{-1}t +\delta) \rfloor}\tilde{R}_k>\epsilon\sqrt{n}\}+\mathsf{P}\{|\nu_{nt}-\theta^{-1}nt|\geq\delta n\}\\&\leq\frac{1}{\epsilon^2 n}\mathsf{E}[\tilde{R}^2_1\mathsf{1}_{\{\tilde{R}_1\geq \epsilon\sqrt{n}\}}]+\frac{\lfloor n(\theta^{-1}t +\delta) \rfloor-1}{\epsilon^2n}\mathsf{E}[\tilde{R}^2_2\mathsf{1}_{\{\tilde{R}_2\geq \epsilon\sqrt{n}\}}]+\mathsf{P}\{|\nu_{nt}-\theta^{-1}nt|\geq\delta n\}\end{align*}Because of \eqref{flln rp} and the assumption of finite second moment for $\tilde{R}_m$, the last term in the previous inequalities goes to zero and hence we get \eqref{resid to zero}.
\end{proof}
\begin{theorem}\label{weak law regenerative processes}
Let $\{\eta_k\}_{k\geq 1}$ be a delayed regenerative process with regeneration epochs $\{\tau_m\}_{m\geq 0}$, $\tau_0 = 0$, such that $\mathsf{E}[\tau_m] < +\infty$ for every $m \in \mathbb{N}$. If, for every $m \in \mathbb{N}$,
\begin{equation}\label{assumption regslln}
\mathsf{E}\biggl[\sum_{k=\tau_{m-1}+1}^{\tau_m}|\eta_k|\biggr] <+\infty.
\end{equation}
Then
\begin{equation}
\lim_{n\to +\infty}\frac{1}{n}\sum_{k=1}^n\eta_k = \frac{\mathsf{E}[\sum_{k=\tau_1+1}^{\tau_2}\eta_k]}{\mathsf{E}[\tau_2-\tau_1]} \;\; a.s..
\end{equation}
\begin{proof}
By assumption, $\{\nu_n\}_{n\geq 0}$ forms a delayed renewal process. Then, the the basic renewal theorem implies
\begin{equation}\label{basic renewal theorem}\lim_{n\to+\infty}\frac{\nu_n}{n} = \frac{1}{\mathsf{E}[\tau_2-\tau_1]} \;\; a.s.
\end{equation}
Now, the normalized partial sum of $\{\eta_k\}_{k\geq 1}$ can be rewritten as
\begin{equation}\label{sample mean}
\frac{1}{n}\sum_{k=1}^n\eta_k = \frac{1}{n}\sum_{k=1}^{\tau_{1}}\eta_k + \frac{1}{n}\sum_{m=2}^{\nu_n}\sum_{\;\;k=\tau_{m-1}+1}^{\tau_{m}}\eta_k+\frac{1}{n}\sum_{k=\tau_{\nu_n}+1}^n\eta_k.
\end{equation}
%By denoting with\begin{equation*}\tilde{R}_m = \sum_{k=\tau_{m-1}+1}^{\tau_{m}}|\eta_k| \;\; m \in \mathbb{N}\end{equation*}
By the regenerative property, for every $m \in \mathbb{N}$, the r.v.s $\sum_{m=1}^{\nu_n}\sum_{k=\tau_{m-1}+1}^{\tau_{m}}\eta_k$, define a sequence of independent random variables that are identically distributed for $m\geq 2$. By the inequality
\begin{equation*}
\bigg|\frac{1}{n}\sum_{k=\tau_{\nu_n}+1}^n\eta_k\bigg| \leq \frac{1}{n}\sum_{k=\tau_{m-1}+1}^{\tau_{m}}|\eta_k|
\end{equation*}
and the first Borel-Cantelli lemma together with \eqref{basic renewal theorem}, we obtain that the first and third terms in \eqref{sample mean} converge almost surely to zero. Finally, the strong law of large numbers and \eqref{basic renewal theorem} lead to the thesis.
\end{proof}
\end{theorem}
%From Theorem \ref{Clt regenerative processes} it follows the weak uniform law of large numbers for delayed regenerative processes.\begin{corollary}\label{weak law regenerative processes}Let $\{\eta_k\}_{k\geq 1}$ be a delayed regenerative process with regeneration epochs $\{\tau_m\}_{m\geq 0}$, $\tau_0 = 0$, such that $\mathsf{E}[\tau_m^2] < +\infty$. for every $m \in \mathbb{N}$. If\begin{equation}\mathsf{E}\biggl[\biggl(\sum_{k=\tau_{m-1}+1}^{\tau_m}|\eta_k-\mu|\biggr)^2\biggr] <+\infty\end{equation}where \begin{equation}\mu = \frac{\mathsf{E}[\sum_{k=\tau_{1}+1}^{\tau_2}\eta_k]}{\mathsf{E}[\tau_2-\tau_1]}\end{equation}then, the following limit holds\begin{equation}\lim_{n \to +\infty}\mathsf{P}\biggl\{\sup_{t \in [0,1]}\biggl|\frac{1}{n}\sum_{k=1}^{\lfloor nt\rfloor}\eta_k-\mu t\biggr|>\epsilon \biggr\} = 0 \;\; \epsilon>0.\end{equation}\end{corollary}
\section{Limit theorems for Markov renewal processes}\label{SezioneProprietaAsintoticheMarkovRinnovo}
In this section, we briefly recall the definition and the main properties of Markov renewal processes. These processes can be thought as a model for the motion of a particle switching from one state to another with random sojourn times in between; the successive states visited form a Markov chain, and the distribution of the sojourn time depends on both the current state and the next state to be entered. Thus, a Markov renewal process generalizes a Markov process by allowing sojourn times that are not necessarily exponentially distributed and may also depend on the next state.

\begin{definition}
Let $(\Omega,\mathcal{G},\mathsf{P})$ a probability space and $(\hat{V},S) = \{\hat{V}_k,S_k\}_{k\geq 0}$ be a stochastic process on $(\Omega,\mathcal{G},\mathsf{P})$, where $\hat{V}_k$ takes values in a countable set $\mathcal{V}$ and $S_k$ are random variables such that $S_{k+1} > S_{k}$, for every $k \in \mathbb{N}_0$ and $S_0 = 0$. Denote with $\mathcal{F} = \{\mathcal{F}_k\}_{k \geq 0}$ the natural filtration associated to the process $(\hat{V},S)$. Then $(\hat{V},S)$ is said to be a Markov renewal process with state space $\mathcal{V}$ provided that
\begin{equation}\label{Markov renewal property}
\mathsf{P}\{\hat{V}_{k+1} = v,S_{k+1}-S_{k} \leq t|\mathcal{F}_{k}\} = \mathsf{P}\{\hat{V}_{k+1} = v,S_{k+1}-S_{k} \leq t|\hat{V}_{k}\},\ \;\; v \in \mathcal{V},\; t\geq 0,\; k \in \mathbb{N}_0.
\end{equation}
We always assume that the Markov renewal process $(\hat{V},S)$ is time-homogeneous. That is, for any $w,v \in \mathcal{V}$, $t \geq 0$,
\begin{equation}
\mathsf{P}\{\hat{V}_{k+1} = w,S_{k+1}-S_{k} \leq t|\hat{V}_{k}=v\} = Q_{vw}(t).
\end{equation}
\end{definition}
From the definition it follows that $\{\hat{V}_n\}_{n\geq 0}$ is a Markov-chain with state space $\mathcal{V}$ and transition matrix $P = \{p_{vw};\; v,w \in \mathcal{V}\}$, where
\begin{equation}
p_{vw} = \lim_{t \to +\infty}Q_{vw}(t).
\end{equation}
The family of probabilities $Q = (Q_{wv}(t);\; w,v \in \mathcal{V}, t \geq 0)$ is called a semi-Markov kernel over $\mathcal{V}$. Moreover, we define the following quantities
\begin{align}
&F_{wv}(t) = \mathsf{P}\{S_{k+1}-S_{k} \leq t|\hat{V}_{k+1} = v,\hat{V}_{k} = w\} = \frac{Q_{wv}(t)}{p_{wv}}\\
&Q_{w \cdot}(t) = \mathsf{P}\{S_{k+1}-S_{k} \leq t|\hat{V}_{k} = w\} = \sum_{v \in \mathcal{V}}Q_{wv}(t).
\end{align}
and we use the notation $F_{wv}(dt) = F_{wv}(t)dt$ and $Q_{w\cdot}(dt) = Q_{w\cdot}(t)dt$ to express the associate measure.

Let $N = (N(t))_{t \geq 0}$ be the counting process associated with $S$, namely
\begin{equation}\label{counting process}
N(t) = \max\{k\in \mathbb{N}_0: S_k \leq t\}.
\end{equation}
The inter-arrivals time of $N$ are denoted as $\xi_k = S_{k}-S_{k-1}$, $\xi_0 = 0$. Then, by \eqref{Markov renewal property} it follows that , for any $n \in \mathbb{N}$, the r.v.s $\xi_1,...,\xi_n$ are conditionally independent given $\hat{V}_0,...,\hat{V}_n$ with the distribution of $\xi_k$ depending only on $\hat{V}_k$ and $\hat{V}_{k-1}$ for $k=1,...,n$. Indeed, for any $t_1,...,t_n \in [0,+\infty)$,
\begin{equation}\label{Xi cond indep MC}
\mathsf{P}\{\xi_1\leq t_1,...,\xi_n\leq t_n|\hat{V}_0,...,\hat{V}_n\} = \prod_{k=1}^n\mathsf{P}\{\xi_k\leq t_k|\hat{V}_k,\hat{V}_{k-1}\}.
\end{equation}
The definition of Markov renewal process implies that the sequence $\{\hat{V}_k,\xi_k\}_{k\geq 0}$ forms a bivariate Markov process with state space $E = \mathcal{V}\times [0,+\infty)$ in which the future state depends from the past only through the present provided by $\hat{V}$. As a consequence of the strong Markov property, it holds that, for any stopping time $\tau$ with respect to $\mathcal{F}$,
\begin{align}\label{Strong Markov renewal property}
&\mathsf{P}\{\hat{V}_{k+1+\tau} = v,S_{k+1+\tau}-S_{k+\tau} \leq t|\mathcal{F}_{\tau}\} = \mathsf{P}\{\hat{V}_{k+1+\tau} = v,S_{k+1+\tau}-S_{k+\tau} \leq t|\hat{V}_{\tau}\}.% \nonumber \\&=\mathsf{P}\{\hat{V}_{k+1} = v,S_{k+1}-S_{k} \leq t|\hat{V}_{0}\} \;\; (v,t) \in E,k \in \mathbb{N}_0 .
\end{align}

We recall a property that will be used frequently in what follows. If $\hat{V}$ is an irreducible Markov chain with finite state space $\mathcal{V}$, the first passage time to the state $v_0 \in \mathcal{V}$,
\begin{equation}
\tau_1 = \inf\{k\geq 1: \hat{V}_k = v_0\},
\end{equation}
satisfies 
\begin{equation}\label{finite moments}
\mathsf{P}\{\tau_1\geq n\}\leq k_1e^{-k_2n}
\end{equation}
for every $n \in \mathbb{N}$ and for some positive constants $k_1,k_2$.

We now study the properties of the sequence of random variables $\{\hat{V}_{k-1},\xi_k\}_{k\geq 1}$ obtained from the Markov renewal process $(\hat{V},S)$. It is well known that an irreducible positive recurrent Markov process forms a regenerative process with respect to the sequence of successive passage times on a fixed state. As the next theorem shows, this property is inherited by the sequence $\{\hat{V}_{k-1},\xi_k\}_{k\geq 1}$.

\begin{theorem}\label{Regenerative property}
Let $(\hat{V},S)$ be a Markov renewal process on a finite state space $\mathcal{V}$ and assume that the Markov chain $\hat{V}$ is irreducible. Fix a state $v_0 \in \mathcal{V}$ and define the sequence of random variables $\{\tau_m\}_{m \geq 0}$, where $\tau_0 = 0$,
\begin{equation}\label{successive passage times}
\tau_m = \inf\{k>\tau_{m-1}: \hat{V}_k = v_0\}.
\end{equation}%\textcolor{red}{If I do not fix the initial state as $v_0$ then $(X,S)$ is a delayed regenerative process}
Then, the sequence of random variables $\{f(\hat{V}_{k-1},\xi_k)\}_{k \geq 1}$, where $f:E \to \mathbb{R}$ is any $\mathcal{B}(E)$-measurable function, is delayed regenerative with respect to the stopping times $\{\tau_m\}_{m\geq 0}$ associated to $\mathcal{F}$. In particular, for any $d,m \in \mathbb{N}$, $1\leq k_1<...<k_d$ and any measurable and bounded function $g:E^d \to \mathbb{R}$
\begin{align}\label{delayed reg}
&\mathsf{E}[g(\hat{V}_{k_1-1+\tau_m},\xi_{k_1+\tau_m},...,\hat{V}_{k_d-1+\tau_m},\xi_{k_d+\tau_m})|\mathcal{F}_{\tau_m}] \nonumber\\
&= \mathsf{E}[g(\hat{V}_{k_1-1},\xi_{k_1},...,\hat{V}_{k_d-1},\xi_{k_d})|\hat{V}_0 = v_0].
\end{align}
\end{theorem}
\begin{proof}
The sequence of random variables $\{\tau_m\}_{m \geq 0}$ by construction forms a sequence of stopping time with respect to $\mathcal{F}$, the filtration associated to $(\hat{V},S)$. Moreover, they represent the times of successive visit of the Markov chain $\hat{V}$ to state $v_0$ and so they form a delayed renewal sequence (since $\hat{V}_0$ may be different from $v_0$). Therefore, we need only to show that condition $(iia)$ of Definition \ref{Regenerative} is satisfied. To this end, it is equivalent to prove that for every $d,m \in \mathbb{N}$, $1\leq k_1<...<k_d$,
\begin{align}\label{proof regenerative}
&\mathsf{E}[\mathsf{1}_A(\hat{V}_{k_1-1+\tau_m},\xi_{k_1+\tau_m},...,\hat{V}_{k_d-1+\tau_{m}},\xi_{k_d+\tau_{m}})|\mathcal{F}_{\tau_m}] \nonumber\\
&=\mathsf{E}[\mathsf{1}_A(\hat{V}_{k_1-1+\tau_1},\xi_{k_1+\tau_1},...,\hat{V}_{k_d-1+\tau_{1}},\xi_{k_d+\tau_{1}})]
\end{align}
for every set $A = \varprod_{r=1}^{d} H_r \times J_r$ where $H_r \subset \mathcal{V}$ and $I_r \subset [0,+\infty)$, $r=1,...,d$. By using \eqref{Strong Markov renewal property} and the time homogeneity of the Markov renewal process, the first member of \eqref{proof regenerative} becomes
\begin{align}
&\mathsf{E}[\mathsf{1}_{A}(\hat{V}_{k_1-1+\tau_m},\xi_{k_1+\tau_{m}}...,\hat{V}_{k_d-1+\tau_{m}},\xi_{k_d+\tau_{m}})|\mathcal{F}_{\tau_m}]\nonumber\\
%&=\mathsf{E}\biggl[\prod_{r=1}^{d}\mathsf{1}_{E\times J_r}(X_{k_r+\tau_m},\xi_{k_r+\tau_{m}})\mathsf{1}_{I_r}(X_{k_r-1+\tau_m})|\mathcal{F}_{\tau_m}\biggr]\nonumber\\
%&=\mathsf{E}[\mathsf{1}_{I_0}(X_{k_1-1+\tau_m})|\mathcal{F}_{\tau_m}]\prod_{r = 1}^{d-1}\mathsf{E}[\mathsf{1}_{I_r\times J_r}(X_{k_r+\tau_m},\xi_{k_r+\tau_{m}})|\mathcal{F}_{k_{r-1}+\tau_{m}})]\mathsf{E}[\mathsf{1}_{J_d}(\xi_{k_d+\tau_m})|\mathcal{F}_{k_{d-1}+\tau_{m}}]\\
%&=\mathsf{E}\biggl[\mathsf{1}_{I_1}(X_{k_1-1+\tau_m})\prod_{r = 1}^{d-1}\mathsf{1}_{E\times J_r}(X_{k_r+\tau_m},\xi_{k_r+\tau_{m}})\mathsf{1}_{I_r}(X_{{k_{r+1}}-1+\tau_m})\mathsf{1}_{E\times J_d}(X_{k_d+\tau_m},\xi_{k_d+\tau_{m}})\bigg|\mathcal{F}_{\tau_m}\biggr]\\
&=\mathsf{E}[\mathsf{1}_{H_1}(\hat{V}_{k_1-1+\tau_m})|\hat{V}_{\tau_m}]\!\!\!\sum_{v_1,...,v_d \in \mathcal{V}}\prod_{p = 1}^{d-1}\mathsf{E}[\mathsf{1}_{\{v_p\}\times J_p}(\hat{V}_{k_p+\tau_m},\xi_{k_p+\tau_m})|\hat{V}_{k_{p}-1+\tau_m}\in H_{p}]\nonumber\\
&\ \ \times\mathsf{E}[\mathsf{1}_{H_p}(\hat{V}_{{k_{p+1}}-1+\tau_m})|\hat{V}_{k_{p}+\tau_m}=v_{p}]\mathsf{E}[\mathsf{1}_{\{v_d\}\times J_d}(\hat{V}_{k_d+\tau_m},\xi_{k_d+\tau_m})|\hat{V}_{k_{d}-1+\tau_m}\in H_d]\nonumber\\
&=\mathsf{E}[\mathsf{1}_{H_1}(\hat{V}_{k_1-1})|\hat{V}_0=v_0]\!\!\!\sum_{v_1,...,v_d \in \mathcal{V}}\prod_{p = 1}^{d-1}\mathsf{E}[\mathsf{1}_{\{v_p\}\times I_p}(\hat{V}_{k_p},\xi_{k_p})|\hat{V}_{k_{p}-1}\in H_{p}]\mathsf{E}[\mathsf{1}_{H_p}(\hat{V}_{{k_{p+1}}-1})|\hat{V}_{k_{p}}=v_{p}]\nonumber\\
&\ \ \times\mathsf{E}[\mathsf{1}_{\{v_d\}\times J_d}(\hat{V}_{k_d},\xi_{k_d})|\hat{V}_{k_{d}-1}\in H_d]\nonumber\\
&=\mathsf{E}\biggl[\mathsf{1}_{H_1}(\hat{V}_{k_1-1})\prod_{p = 1}^{d-1}\mathsf{1}_{\mathcal{V}\times I_p}(\hat{V}_{k_p},\xi_{k_p})\mathsf{1}_{H_p}(\hat{V}_{{k_{p+1}}-1})\mathsf{1}_{\mathcal{V}\times I_d}(\hat{V}_{k_d},\xi_{k_d})\bigg|\hat{V}_0=v_0\biggr]\nonumber\\
&=\mathsf{E}[\mathsf{1}_A(\hat{V}_{k_1-1},\xi_{k_1},...,\hat{V}_{k_d-1},\xi_{k_d})|\hat{V}_0 =v_0] \label{proof regen 1}.
\end{align}
By looking at the second member of \eqref{proof regenerative}, from the fact that, for every $m \in \mathbb{N}$, $\{\hat{V}_{\tau_m} = v_0\}$ is an almost sure event, by applying once again \eqref{Strong Markov renewal property} and the time homogeneity of the Markov renewal process, it follows
\begin{align}
\mathsf{E}[\mathsf{1}_A(\hat{V}_{k_1-1+\tau_1},\xi_{k_1+\tau_1},...,\hat{V}_{k_d-1+\tau_{1}},\xi_{k_d+\tau_{1}})]
%&=\mathsf{E}[\mathsf{1}_{I_1}(X_{k_1-1+\tau_1})|X_{\tau_1}]\prod_{r = 1}^{d}\mathsf{E}[\mathsf{1}_{E\times J_r}(X_{k_r+\tau_1},\xi_{k_r+\tau_1})|X_{k_{r}-1+\tau_1}]\mathsf{E}[\mathsf{1}_{I_r}(X_{{k_{r+1}}-1+\tau_1})|{X}_{k_{r}+\tau_1}]\nonumber\\&=\mathsf{E}_{v_0}[\mathsf{1}_{I_1}(X_{k_1-1})]\prod_{r = 1}^{d}\mathsf{E}[\mathsf{1}_{E\times J_r}(X_{k_r},\xi_{k_r})|X_{k_{r}-1}]\mathsf{E}[\mathsf{1}_{I_r}(X_{{k_{r+1}}-1})|{X}_{k_{r}}]\nonumber\\&=\mathsf{E}_{v_0}[\mathsf{1}_{I_1}(X_{k_1-1})]\prod_{r = 2}^{d}\mathsf{E}[\mathsf{1}_{I_r\times J_r}(X_{{k_r}-1},\xi_{k_r})|X_{k_{r-1}}]\nonumber\\
&= \mathsf{E}[\mathsf{1}_A(\hat{V}_{k_1-1+\tau_1},\xi_{k_1+\tau_1},...,\hat{V}_{k_d-1+\tau_{1}},\xi_{k_d+\tau_{1}})|\hat{V}_{\tau_1}]\nonumber\\&=\mathsf{E}[\mathsf{1}_A(\hat{V}_{k_1-1},\xi_{k_1+},...,\hat{V}_{k_d-1},\xi_{k_d})|\hat{V}_0=v_0] \label{proof regen 2}.
\end{align}
In conclusion, from \eqref{proof regen 1} and \eqref{proof regen 2} we obtain \eqref{proof regenerative} and hence the thesis.
\end{proof}
We observe that, from \eqref{delayed reg}, it follows that, for any $d \in \mathbb{N}$, $m \in \mathbb{N}_0$ and measurable and bounded measurable function $g:E^{d}\to \mathbb{R}$,
\begin{align}\label{delayed reg2}
&\mathsf{E}[g(\hat{V}_{\tau_{m}},\xi_{\tau_{m}+1},...,\hat{V}_{\tau_{m+1}-1},\xi_{\tau_{m+1}})\mathsf{1}_{\{d\}}(\tau_{m+1}-\tau_{m})]%\nonumber\\&=\mathsf{E}[f(\hat{V}_{\tau_{m}-1},\xi_{\tau_{m}},...,\hat{V}_{\tau_{m+1}-1},\xi_{\tau_{m+1}})\mathsf{1}_{\{d\}}(\tau_{m+1}-\tau_{m})|\mathcal{F}_{\tau_{m}}]
\nonumber\\
&=\mathsf{E}[g(\hat{V}_{0},\xi_{1},...,\hat{V}_{\tau_{1}-1},\xi_{\tau_{1}}))\mathsf{1}_{\{d\}}(\tau_{1})|\hat{V}_0 = v_0].
\end{align}

If we assume that the finite Markov chain $\hat{V}$ is irreducible and aperiodic, then there exists a unique invariant measure $\pi= \{\pi_v; \;\; v \in \mathcal{V}\}$, such that $\pi_{v} > 0$ for every $v \in \mathcal{V}$. The next theorem shows that the $n$-steps probability with fixed length of the first sojourn times, $p^{(n)}_{t;wv}= \mathsf{P}\{\hat{V}_n = v|\hat{V}_0 = w, S_1 = t\}$, converges to the invariant measure $\pi$ at an exponential rate.

\begin{theorem}\label{ergodic theorem Mr}
Let $(\hat{V},S)$ be a Markov renewal process on a finite state space $\mathcal{V}$ and assume that the Markov chain $\hat{V}$ is irreducible and aperiodic. 
% Denote with $\pi = \{\pi_v; \;\; v \in \mathcal{V}\}$ the invariant distribution of $\hat{V}$ and with $p^{(n)}_{t;wv}= \mathsf{P}\{\hat{V}_n = v|\hat{V}_0 = w, S_1 = t\}$. 
Then, the following holds
\begin{equation}
|p^{(n)}_{t;wv} - \pi_v| \leq 3c\rho^{n-1},
\end{equation}
where $c \geq 0$ and $\rho \in [0,1)$.
\end{theorem}
\begin{proof}
The hypothesis of Theorem 8.9 of \cite{Pm Billingsley} are satisfied, therefore there exist $c\geq 0$, $\rho \in [0,1)$ such that 
\begin{equation}\label{billingsley}
|p^{(n)}_{wv} - \pi_v| \leq c\rho^{n},
\end{equation}
where $p^{(n)}_{wv}$ is the probability of transition in $n$ steps of the Markov chain $\hat{V}$. Then we have
\begin{align*}
|p^{(n)}_{t;wv}-\pi_v|& \leq |p^{(n)}_{wv}-\pi_v| + \frac{p^{(n)}_{wv}}{Q_{w\cdot}(dt)}|\mathsf{P}\{S_1 \in dt|\hat{V}_{0} = w,\hat{V}_{n} = v\}-Q_{w\cdot}(dt)| \\
&=|p^{(n)}_{wv}-\pi_v| + \frac{1}{Q_{w\cdot}(dt)}|\sum_{v' \in \mathcal{V}}p_{wv'}F_{wv'}(dt)(p^{(n-1)}_{v'v}-p^{(n)}_{wv})|\leq3c\rho^{n-1},
\end{align*}
where in the last step we made use of the triangular inequality and \eqref{billingsley}. 
\end{proof}
From Theorem \ref{ergodic theorem Mr} we have that, for every $A \subset [0,+\infty)$, $v,w \in \mathcal{V}$ and $t \in [0,+\infty)$
\begin{align*}
\lim_{n\to +\infty}\mathsf{P}\{\hat{V}_{n-1} = v,\xi_n \in A|\hat{V}_0 = w,\xi_1 = t\} = \pi_v\int_A Q_{\cdot v}(dx).
\end{align*}
We now prove an ergodic theorem for the sequence $\{\hat{V}_{k-1},\xi_k\}_{k\geq 1}$. Let $E^{\infty}$ be the product space of $E$. An element of $E$ can be considered as an infinite sequence
\begin{equation*}
(v,x) = \{(Z_0(v),Z'_0(x)),(Z_2(v),Z'_2(x)),...\}
\end{equation*}
where $Z_k:E^{\infty}\to \mathcal{V}$ and $Z'_k:E^{\infty}\to [0,+\infty)$, $k \in \mathbb{N}_0$, are the natural projection functions. A cylinder of rank $n$ is a set of $E^{\infty}$ of the form
\begin{equation*}
C_n = \{(v,x):(Z_0(v),Z'_0(x)) \in \{v_1\}\times I_1,...,(Z_{n-1}(v),Z'_{n-1}(x)) \in \{v_n\}\times I_n\}
\end{equation*}
where $v_1,...,v_n \in \mathcal{V}$ and $I_1,...,I_n$ are intervals of $[0+\infty)$. Let $\mathcal{C}$ be the $\sigma$-field generated by the class of cylinders of all ranks. Through the paper we use the notation $\mathsf{E}_{\pi}$ for the expected value computed when the distribution of $\hat{V}_0$ is the invariant one.
\begin{theorem}[Ergodic Theorem]\label{ergodic theorem}
Let $(\hat{V},S)$ be a Markov renewal process on a finite state space $\mathcal{V}$ and assume that the Markov chain $\hat{V}$ is irreducible and aperiodic with invariant distribution $\pi$. Let $f:E\to \mathbb{R}$ be any $\mathcal{B}(E)$-measurable function such that
\begin{equation}
\sum_{v \in \mathcal{V}}\int_{0}^{+\infty}|f(v,x)|\pi_v Q_{v\cdot}(dx) < +\infty.
\end{equation}
Then,
\begin{equation}
\lim_{n\to +\infty}\frac{1}{n}\sum_{k=1}^nf(\hat{V}_{k-1},S_k-S_{k-1}) = \mathsf{E}_\pi[f(\hat{V}_0,S_1)] %= \sum_{v \in \mathcal{V}}\int_{0}^{+\infty}f(v,x)\pi_v Q_{v\cdot}(dx)
\;\;a.s..
\end{equation}
\end{theorem}
\begin{proof}
Let $T:E^{\infty}\to E^{\infty}$ the shift operator defined on the measurable space $(E^{\infty},\mathcal{C})$. That is 
\begin{equation*}
\bigl(Z_k(T(v,x)),Z'_k(T(v,x))\bigr) =  \bigl(Z_{k+1}(v),Z'_{k+1}(x)\bigr)
\end{equation*}
for every $k \in \mathbb{N}_0$. Define $\zeta:\Omega \to E^{\infty}$ by 
\begin{equation*}
\bigl(Z_{k-1}(\zeta(\omega)),Z'_k(\zeta(\omega))\bigr) = \bigl(\hat{V}_{k-1}(\omega),\xi_k(\omega)\bigr),\ \;\; k \in \mathbb{N}.
\end{equation*}
If $\pi$ is chosen as the initial distribution, then the sequence $\{(\hat{V}_{k-1},\xi_k)\}_{k\geq 1}$ is stationary and hence the shift $T$ preserves $\mathsf{P}\circ \zeta^{-1}$, the distribution of $\{(\hat{V}_{k-1},\xi_k)\}_{k\geq 1}$. Now, let $A$ and $B$ two cylinder sets of rank $k$ and $r$ respectively, $A = \{(v,x):(Z_0(v),Z'_1(x)) \in \{v_0\}\times I_1,...,(Z_{k-1}(v),Z'_k(x)) \in \{v_{k-1}\}\times I_k\}$ and $B = \{(v,x):(Z_0(v),Z'_1(x)) \in \{w_0\}\times J_1,...,(Z_{r-1}(v),Z'_r(x)) \in \{w_{r-1}\}\times J_r\}$, and set $I = I_1\times...\times I_k$, $J=J_1\times...\times J_r$. Then,
\begin{align*}
&\mathsf{P}\circ \zeta^{-1}\{A\cap T^{-(k+n-1)}(B)\}\\
&=\pi_{v_0}\int_I\int_J\prod_{i=1}^{k-1}Q_{v_{i-1}v_{i}}(dt_i)Q_{v_{k-1}\cdot}(dt_k)p^{(n)}_{t_k;v_kw_0}\prod_{j=1}^{r-1}Q_{w_{j-1}w_{j}}(ds_j)Q_{w_{r-1}\cdot}(ds_r).
\end{align*}
By the result in Theorem \ref{ergodic theorem Mr} it follows that
\begin{equation*}
\lim_{n\to+\infty}\mathsf{P}\circ \zeta^{-1}\{A\cap T^{-(k+n-1)}(B)\} = \mathsf{P}\circ \zeta^{-1}\{A\}\mathsf{P}\circ \zeta^{-1}\{B\},
\end{equation*}
that is, the shift operator is mixing. Therefore, the shift operator is ergodic under $\mathsf{P}\circ\zeta^{-1}$ (see Theorem 1.17 of \cite{Walters}), and hence the thesis follows by applying the Birkhoff Ergodic Theorem.
\end{proof}
Another asymptotic behavior exhibited by the sequence $\{\hat{V}_{k-1},\xi_k\}_{k\geq 1}$, which can be deduced from the embedded Markov chain $\hat{V}$, is the mixing property.

\begin{definition}Let $\{\eta_k\}_{k \geq 0}$ be a stationary sequence of random variables and denote with $\mathcal{F}_k = \sigma(\eta_0,...,\eta_k)$ and $\mathcal{G}_{k+n} = \sigma(\eta_{k+n},\eta_{k+n+1},...)$. If there exist a number $\varphi_n$ defined as\begin{equation}\varphi_n = \sup\{|\mathsf{P}\{B|A\}-\mathsf{P}\{B\}|: A \in \mathcal{F}_k, \mathsf{P}\{A\}>0, B \in \mathcal{G}_{k+n}\}\end{equation}such that $\lim_{n\to +\infty}\varphi_n = 0$, then the sequence $\{\eta_k\}_{k \geq 0}$ is said to be $\varphi$-mixing.\end{definition}
\begin{theorem}\label{phi mixing}
Let $(\hat{V},S)$ be a Markov renewal process on a finite state space $\mathcal{V}$ and assume that the Markov chain $\hat{V}$ is irreducible and aperiodic. Let $f:E \to \mathbb{R}$ be any $\mathcal{B}(E)$-measurable function. Then the sequence of random variables $\{f(\hat{V}_{k-1},\xi_k)\}_{k \geq 1}$ is $\varphi$-mixing with $\varphi_k = K\rho^{k-1}$, where $K$ is a positive constant and $\rho \in [0,1)$.
\end{theorem}
\begin{proof}
By assumptions there exist a stationary distribution $\pi = \{\pi_v; \;\; v \in \mathcal{V}\}$ for the Markov chain $\hat{V}$. If the initial distribution is the stationary one, then clearly the sequence $\{f(\hat{V}_{k-1},\xi_k)\}_{k \geq 1}$ is stationary. Set $A = \{\omega: (\hat{V}_0,\xi_1,...,\hat{V}_{k-1},\xi_k)(\omega) \in \{v_0\}\times I_1 \times...\times\{v_{k-1}\}\times I_k\}$ and  $B = \{\omega :(\hat{V}_{k+n-1},\xi_{k+n},...,\hat{V}_{k+n+r-1},\xi_{k+n+r})(\omega) \in \{w_0\}\times J_1 \times...\times\{w_{k-1}\}\times J_k\}$, for $v_i,w_j \in \mathcal{V}$, and intervals $I_i,J_j \subset [0,+\infty)$, $i=0,...,k-1$, $j = 0,...,r-1$. Then, we have that
\begin{align*}
&|\mathsf{P}\{A \cap B\}-\mathsf{P}\{A\}\mathsf{P}\{B\}| \leq \\
&\int_{I}\int_{J}\pi_{v_0}\prod_{i=1}^{k-1}Q_{v_{i-1}v_{i}}(dt_i)Q_{v_{k-1}\cdot}(dt_{k})\prod_{j=1}^{r-1}Q_{w_{j-1}w_{j}}(ds_{j})Q_{w_{r-1} \cdot }(ds_{r})|p^{(n)}_{t_k;v_{k-1}w_0}-\pi_{w_0}|
\end{align*}
where $I = I_1\times...\times I_k$, $J=J_1\times...\times J_r$. By applying Theorem \ref{ergodic theorem Mr}, it follows that
\begin{align*}
|\mathsf{P}\{A \cap B\}-\mathsf{P}\{A\}\mathsf{P}\{B\}| &\leq
\int_{I}\int_{J}\pi_{v_0}\prod_{i=1}^{k-1}Q_{v_{i-1}v_{i}}(dt_k)Q_{v_{k-1}\cdot}(dt_{k})\prod_{j=1}^{r-1}Q_{w_{j-1}w_{j}}(ds_{j})Q_{w_{r-1} \cdot }(ds_{r})3c\rho^{n-1}\\
&\leq 3c\#(\mathcal{V})\mathsf{P}\{A\}\rho^{n-1}
\end{align*}
where $\#$ denotes the cardinality of a set.
Thus, we can conclude that
\begin{equation}\label{mixing}
|\mathsf{P}\{B|A\}-\mathsf{P}\{B\}| \leq K\rho^{n-1}.
\end{equation}
with $K = 3c\#(\mathcal{V})$. By a classical extension argument, the class of sets $A$ and the class of sets $B$ satisfying \eqref{mixing} can be extended respectively to $\sigma(\hat{V}_0,\xi_1,...,\hat{V}_{k-1},\xi_{k})$ and $\sigma(\hat{V}_{k+n-1},\xi_{k+n},...)$, from which we have the thesis.
\end{proof}
Theorem \ref{Regenerative property} implies that the sequence $\{\xi_k\}_{k\geq1}$ is delayed regenerative (it suffices to take the projection onto the second component). We use this property to prove a basic renewal-type theorem for $N$, the counting process associated to the Markov renewal process $(\hat{V},S)$.
\begin{theorem}\label{renewal theorem}
Let $(\hat{V},S)$ be a Markov renewal process on a finite state space $\mathcal{V}$ and assume that the Markov chain $\hat{V}$ is irreducible and that $\mathsf{E}[S^2_1|\hat{V}_0 = v,V_1 = w]<+\infty$ for every $v,w \in \mathcal{V}$. For $v_0 \in \mathcal{V}$, set with
\begin{equation}
\mu = \frac{\mathsf{E}[\sum_{k=1}^{\tau_1}\xi_k|\hat{V}_0 = v_0]}{\mathsf{E}[\tau_1|\hat{V}_0 = v_0]}.
\end{equation}
where $\tau_1 = \inf\{k>1: \hat{V}_k = v_0\}$. Then, for every $T>0$,
\begin{equation}
\lim_{n \to +\infty}\mathsf{P}\biggl\{\sup_{t \in [0,T]}\biggl|\frac{N(n t)}{n}-\frac{t}{\mu}\biggr| \geq \epsilon \biggr\} = 0,\ \;\; \epsilon >0,
\end{equation}
\end{theorem} 
\begin{proof}
By Theorem \ref{Regenerative property}, the sequence of random variable $\{\xi_k\}_{k\geq 0}$ is delayed regenerative with respect to the sequence of stopping times $\{\tau_m\}_{m\geq 0}$, $\tau_0 = 0$, defined in \eqref{successive passage times}. The irreducibility of the chain and the finiteness of the state space $\mathcal{V}$ imply through \eqref{finite moments} that $\mathsf{E}[\tau^2_m]<+\infty$ and $\mathsf{E}[\tau_m-\tau_{m-1}] = \mathsf{E}[\tau_1|\hat{V}_0 = v_0]$ for every $m \in \mathbb{N}$. Furthermore, by means of \eqref{Xi cond indep MC} and \eqref{delayed reg2}, we can write, for every $m \in \mathbb{N}$,
\begin{align}\label{assumption reg}
\mathsf{E}\biggl[\biggl(\sum_{k=\tau_{m-1}+1}^{\tau_{m}}\xi_k\biggr)^2\biggr] &= \mathsf{E}\biggl[\biggl(\sum_{k=1}^{\tau_{1}}\xi_k\biggr)^2\bigg|\hat{V}_0 = v_0\biggr] \nonumber\\
&= \sum_{n\geq 1}\sum_{\substack{v_1,...,v_{n-1} \\ \in\mathcal{V}\setminus \{v_0\}}}\mathsf{P}\{\hat{V}_1=v_1,...,\hat{V}_n = v_0|\hat{V}_0 = v_0\}\mathsf{E}\biggl[\sum_{k=1}^{n}\xi_k^2+\sum_{k\neq k'}\xi_k\xi_{k'}\bigg|\hat{V}_0 = v_0,...,\hat{V}_n = v_0\biggr]\nonumber\\
&=\sum_{n\geq 1}\sum_{\substack{v_1,...,v_{n-1} \\ \in\mathcal{V}\setminus \{v_0\}}}\mathsf{P}\{\hat{V}_1=v_1,...,\hat{V}_n = v_0|\hat{V}_0 = v_0\}\biggl\{\sum_{k=1}^{n}\mathsf{E}[\xi_1^2|\hat{V}_1 = v_{k},\hat{V}_0 = v_{k-1}]\nonumber\\
&\ \ +\sum_{k\neq k'}\mathsf{E}[\xi_1^2|\hat{V}_1 = v_{k},\hat{V}_0 = v_{k-1}]\mathsf{E}[\xi_1^2|\hat{V}_1 = v_{k'},\hat{V}_0 = v_{k'-1}]\biggr\}\nonumber\\
&\leq\max_{v,v' \in \mathcal{V}}\mathsf{E}[\xi^2_1|\hat{V}_1 = v',\hat{V}_0 = v]\sum_{n\geq1}n\mathsf{P}\{\tau_1=n|\hat{V}_0 = v_0\}\nonumber\\
&\ \ +\max_{v,v' \in \mathcal{V}}(\mathsf{E}[\xi_1|\hat{V}_1 = v',\hat{V}_0 = v])^2\sum_{n\geq1}n(n-1)\mathsf{P}[\tau_1=n|\hat{V}_0 = v_0]\nonumber\\
&\leq \max_{v,v' \in \mathcal{V}}\mathsf{E}[\xi^2_1|\hat{V}_1 = v',\hat{V}_0 = v]\mathsf{E}[\tau^2_1|\hat{V}_0 = v_0],
\end{align}
where we set in the notation above $v_n = v_0$. Thus, $\mathsf{E}[(\sum_{k=\tau_{m-1}+1}^{\tau_{m}}\xi_k)^2]<+\infty$ for every $m \in \mathbb{N}_0$. Now, we can apply Theorem \ref{Clt regenerative processes} and it follows that the process $\bar{\xi}_\lambda$ defined by
\begin{equation*}
\bar{\xi}_n(t) = \frac{1}{\sqrt{n}\sigma}\sum_{k=1}^{\lfloor n t\rfloor}(\xi_k-\mu),\ \;\; t \in [0,T]
\end{equation*}
where $\sigma^2 = (\mathsf{E}[\tau_1|\hat{V}_0 = v_0])^{-1}\mathsf{E}[(\sum_{k=1}^{\tau_1}(\xi_k-\mu))^2|\hat{V}_0 = v_0]$, satisfies $\bar{\xi}_n \Rightarrow W$ in $D[0,T]$. Now, by Theorem 14.6 of Billingsley \cite{Billingsley}, the normalized counting process $\bar{N}_n$,
\begin{align*}
\bar{N}_n(t) = \frac{N(n t)-\mu^{-1}n t}{\sigma\mu^{-\frac{3}{2}}\sqrt{n}},\ \;\; t \in [0,T],
\end{align*}
satisfies $\bar{N}_n \Rightarrow W$ on $D[0,T]$. Finally, the continuous mapping theorem leads to
\begin{equation*}
\sup_{t \in [0,T]}\biggl|\frac{1}{\sqrt{n}}\frac{N(n t)-\mu^{-1}n t}{\sigma\mu^{-\frac{3}{2}}\sqrt{n}}\biggr| \Rightarrow 0,
\end{equation*}
and hence the thesis.
\end{proof}
Let us define the residual life of the process $N$ as
\begin{equation}
R(t) = t-S_{N(t)}.
\end{equation}
The next theorem describes the asymptotic behavior of residual life $R$.
\begin{theorem}\label{Residual life}
Let $(\hat{V},S)$ a Markov renewal process on a finite state space $\mathcal{V}$ and assume that the Markov chain $\hat{V}$ is irreducible and that $\mathsf{E}[S^2_1|\hat{V}_0 = v,V_1 = w]<+\infty$ for every $v,w \in \mathcal{V}$. Then 
\begin{equation}
\lim_{n\to +\infty}\mathsf{P}\biggl\{\sup_{t \in [0,T]}\frac{R(nt)}{\sqrt{n}}>\epsilon\biggr\} = 0,\ \;\; \epsilon>0.
\end{equation}
\end{theorem}
\begin{proof}
By definition of $N$, we have the following inequality
\begin{equation*}\label{xi random index divided by sqrt lambda goes to 0}
\sup_{t \in [0,T]}\frac{n t-S_{N(n t)}}{\sqrt{n}} \leq \sup_{t \in [0,T]}\frac{\xi_{N(n t)+1}}{\sqrt{n}} = \max_{k = 1,...,N(n T)}\frac{\xi_{k+1}}{\sqrt{n}}.
\end{equation*}
By assumptions there exist a stationary distribution for the Markov chain $\hat{V}$ such that the random variable forming the sequence $\{\xi_k\}_{k\geq 1}$ have the same distribution and finite second moment. Therefore, we can apply Lemma \ref{lemma 3} and by Theorem \ref{renewal theorem} we obtain the thesis.
\end{proof}
In the next theorem is provided a Wald's type identity related to the sum $\sum_{k=1}^{\tau_1}f(\hat{V}_{k-1},S_k-S_{k-1})$, where $\tau_1 = \inf\{k>1: \hat{V}_k = v_0\}$, for some $v_0 \in \mathcal{V}$.
\begin{theorem}\label{prop 1}
Let $(\hat{V},S)$ be a Markov renewal process on a finite state space $\mathcal{V}$ and assume that the Markov chain $\hat{V}$ is irreducible and aperiodic with invariant distribution $\pi$.  Let $f:E\to \mathbb{R}$ be any $\mathcal{B}(E)$-measurable function such that, for every $v,w \in \mathcal{V}$,
\begin{equation}\label{assumption prop 1}
\int_{0}^{+\infty}|f(v,x)|F_{vw}(dx) < +\infty.
\end{equation}
Then, 
\begin{equation}
\mathsf{E}\bigg[\sum_{k=1}^{\tau_1}f(\hat{V}_{k-1},S_k-S_{k-1})\bigg|\hat{V}_0 = v_0\bigg] = \mathsf{E}[\tau_1|\hat{V}_0 = v_0]\mathsf{E}_\pi[f(\hat{V}_0,S_1)].
\end{equation}    
\end{theorem}
\begin{proof}
From Theorem \ref{Regenerative property}, the sequence $\{f(\hat{V}_{k-1},S_k-S_{k-1})\}_{k \geq 1}$ is delayed regenerative with respect to the sequence of stopping times $\{\tau_m\}_{m\geq 0}$, $\tau_0 = 0$, defined in \eqref{successive passage times}. The irreducibility of the chain and the finiteness of the state space $\mathcal{V}$ implies that $\mathsf{E}[\tau_m]<+\infty$ and $\mathsf{E}[\tau_m-\tau_{m-1}] = \mathsf{E}[\tau_1|\hat{V}_0 = v_0]$ for every $m \in \mathbb{N}$. From equation \eqref{delayed reg2} we have, for every $m \in \mathbb{N}$,
\begin{align*}
\mathsf{E}\bigg[\sum_{k=\tau_{m-1}+1}^{\tau_m}&|f(\hat{V}_{k-1},S_k-S_{k-1})|\bigg] = \mathsf{E}\biggl[\sum_{k=1}^{\tau_1}|f(\hat{V}_{k-1},\xi_k)|\bigg|\hat{V}_0 = v_0\biggr] \\
&=\sum_{n\geq 1}\sum_{\substack{v_1,...,v_{n-1} \\ \in\mathcal{V}\setminus v_0}}\mathsf{P}\{\hat{V}_1=v_1,...,\hat{V}_n = v_0|\hat{V}_0 = v_0\}\mathsf{E}\biggl[\sum_{k=1}^{n}|f(\hat{V}_{k-1},\xi_k)|\bigg|\hat{V}_0 = v_0,...,\hat{V}_n = v_0\biggr] \\
&=\sum_{n\geq 1}\sum_{\substack{v_1,...,v_{n-1} \\ \in\mathcal{V}\setminus v_0}}\mathsf{P}\{\hat{V}_1=v_1,...,\hat{V}_n = v_0|\hat{V}_0 = v_0\}\sum_{k=1}^n\mathsf{E}[|f(\hat{V}_{0},\xi_1)||\hat{V}_0 = v_{k-1},\hat{V}_1 = v_k] \\
&\leq \mathsf{E}[\tau_1|\hat{V}_0 = v_0]\max_{v,w \in \mathcal{V}}(\mathsf{E}[|f(\hat{V}_{0},\xi_1)||\hat{V}_1 = w,\hat{V}_0 = v]) <+\infty,
\end{align*} 
where we set $v_n = v_0$ and used \eqref{Xi cond indep MC}. Hence, by Theorem \ref{weak law regenerative processes} 
\begin{equation*}
\lim_{n\to+\infty}\frac{1}{n}\sum_{k=1}^nf(\hat{V}_{k-1},S_k-S_{k-1}) = \frac{\mathsf{E}[\sum_{k=1}^{\tau_1}f(\hat{V}_{k-1},S_k-S_{k-1})|\hat{V}_0 = v_0]}{\mathsf{E}[\tau_1|\hat{V}_0 = v_0]} \;\; a.s.
\end{equation*}
Moreover, the sequence $\{f(\hat{V}_{k-1},S_k-S_{k-1})\}_{k\geq 1}$ is also stationary if $\hat{V}_0$ has distribution given by $\pi$. The finiteness of $\mathcal{V}$ implies that \eqref{assumption prop 1} is equivalent to
\begin{equation*}
\sum_{v \in \mathcal{V}}\int_{0}^{+\infty}|f(v,x)|\pi_v Q_{v\cdot}(dx) < +\infty.
\end{equation*}
Therefore, from Theorem \ref{ergodic theorem} it follows that 
\begin{equation*}
\lim_{n\to+\infty}\frac{1}{n}\sum_{k=1}^nf(\hat{V}_{k-1},S_k-S_{k-1}) = \mathsf{E}_\pi[f(\hat{V}_0,S_1)] \;\; a.s.
\end{equation*}
from which we have the thesis.
\end{proof}
We now consider the continuous time process associated to the Markov renewal sequence.
\begin{definition}
The stochastic process $V = (V(t))_{t \geq 0}$ defined by
\begin{equation}
V(t) = 
\begin{cases}
\hat{V}_k &\textit{if}\;\; S_k \leq t < S_{k+1},\\
\Delta &\textit{if}\;\; t \geq \sup_k S_k
\end{cases}
\end{equation}
where $\Delta$ is a point not in $\mathcal{V}$, is known as the semi-Markov process associated with the Markov renewal process $(\hat{V},S)$.
\end{definition}
If we consider a finite state space $\mathcal{V}$, it holds that $\sup_k S_k = +\infty$ (see \cite{Cinlar} p. 327). Then, the semi-Markov process $V$ can be expressed as $V(t) = \hat{V}_{N(t)}$, that is more explicitly,
\begin{equation}
V(t) = \sum_{k\geq1}\hat{V}_{k-1}\mathsf{1}_{[S_{k-1},S_k)}(t) = \sum_{k\geq1}\hat{V}_{k-1}\mathsf{1}_{\{k-1\}}(N(t)).
\end{equation}

In the next we will use the following theorem (see \cite{Cinlar} Theorem (5.22)) that describes the asymptotic behavior of the law of a semi-Markov process $V$.
\begin{theorem}\label{Theorem Cinlar}
Let $V$ be the semi-Markov process associated with the Markov renewal process $(\hat{V},S)$. If the embedded Markov chain $\{\hat{V}_k\}_{k\geq 0}$ is irreducible positive recurrent and aperiodic. Then, for any $v,w \in \mathcal{V}$,
\begin{equation}
\lim_{t \to+\infty}\mathsf{P}\{V(t) = v|V(0) = w\} = \frac{\pi_v \mathsf{E}[S_1|\hat{V}_0 = v]}{\sum_{w\in \mathcal{V}} \pi_w \mathsf{E}[S_1|\hat{V}_0 = w]}
\end{equation}
where $\pi$ is the stationary measure of $\{\hat{V}_k\}_{k\geq 0}$.
\end{theorem}
Under the hypotheses of Theorem \ref{Theorem Cinlar}, if we also assume that the state space $\mathcal{V}$ is finite, then it follows
\begin{equation}\label{limit of the mean of semi-Markov}
\lim_{t \to +\infty}\mathsf{E}[V(t)] = \frac{\mathsf{E}_\pi[\hat{V}_0S_1]}{\mathsf{E}_\pi[S_1]}.
\end{equation}
\section{Weak convergence of the integral of semi-Markov process}\label{sezioneConvergenzaDeboleIntegraleSemiMarkov}
Let $(\Omega,\mathcal{G},\mathsf{P})$ be a probability space where is defined $(\hat{V},S):\Omega\to E^{\infty}$ a homogeneous Markov renewal process with semi-Markov kernel $Q$, where $E^{\infty}$ is the product space of $E=\mathcal{V}\times[0,+\infty)$, with $\mathcal{V}$ the finite state space of $\hat{V}$. We endow $E$ with the $\sigma$-field generated by the cylinder sets $\mathcal{C}$. We assume that the finite Markov chain $\hat{V}$ is irreducible and aperiodic with unique invariant measure $\pi$. Let $\xi:\Omega \to [0,+\infty)^{\infty}$ defined as $\xi_k = S_k -S_{k-1}$ with $k \in \mathbb{N}_0$, $\xi_0 = 0$. Let $N:\Omega\to D[0,+\infty)$ be the counting process associated to $S$, $N(t) = \max\{k \in \mathbb{N}_0:S_k \leq t\}$ $t\geq 0$, where $D[0,+\infty)$ is the space of cadlag functions defined on $[0,+\infty)$ endowed by the Skorohod topology, and $V:\Omega\to D[0,+\infty)$ the semi-Markov process associated to $(\hat{V},S)$, \textit{i.e.} $V(t) = \hat{V}_{N(t)}$. Let $X:\Omega\to C[0,+\infty)$ be the integral of $V$, where $C[0,+\infty)$ is the space of continuous functions on $[0,+\infty)$ endowed with the topology of the uniform metric,
\begin{equation}\label{integral of V}
X(t) = \int_0^tV(s)ds,\ \;\; t \geq 0.
\end{equation}
The stochastic process $X$ can be thought as the position of a particle moving with velocity $V$. This means that the motion performs displacements of velocity given by $\hat{V}$ for a random amount of time of length determined by $\xi$ that depends on both the current velocity and the following one. Hence, each element of the sequence of random variables $\{\hat{V}_{k-1},\xi_k\}_{k\geq 1}$ contains the information about the velocity of the $k$-th displacement and the duration of the time interval in which the motion moves according to this velocity. 
\\By observing that the time interval $[0,t]$ can be partitioned with respect to the sequence of the arrival times $S$, we obtain the following equivalent representation of $X$ 
\begin{equation}\label{random walk plus interpolation}
X(t)=\sum_{k=1}^{N(t)}\hat{V}_{k-1}(S_{k} -S_{k-1})+\hat{V}_{N(t)}( t- S_{N(t)}),\ \;\; t \geq 0.
\end{equation}

The aim of the present section is to identify conditions under which $X$ admits a weak limit. It turns out that the suitable reparameterization of $V$ is the following. By introducing a parameter $\lambda>0$, we define the scaled counting process $N_\lambda = (N_\lambda (t))_{t\geq 0}$ as
\begin{equation}\label{scaled counting process}
N_\lambda(t) = \max\{k \in \mathbb{N}_0: \lambda^{-1}S_k \leq t\} = N(\lambda t).
\end{equation}
Then, we normalize the semi-Markov process $V(\lambda t)$ associated to $N(\lambda t)$ by defining the new process $V_\lambda = (V_{\lambda}(t))_{t\geq 0}$ as
\begin{equation}\label{normalized semi-markov}
V_{\lambda}(t) = \sqrt{\lambda}(V(\lambda t)-\theta),\ \;\; t \geq 0,
\end{equation}
where, according to \eqref{limit of the mean of semi-Markov}, the parameter $\theta$ is the limit as $\lambda\to+\infty$ of the mean of $V(\lambda t)$
\begin{equation}\label{theta}
\theta  =\lim_{\lambda \to +\infty}\mathsf{E}[V(\lambda t)]=\frac{\mathsf{E}_\pi[\hat{V}_0S_{1}]}{\mathsf{E}_\pi[S_1]}.
\end{equation}
The corresponding integral of the normalized semi-Markov process $V_\lambda$ is denoted as $X_\lambda$ and takes the form
\begin{equation}
X_\lambda(t) = %\int_0^{t} V_\lambda(s)ds = 
\frac{1}{\sqrt{\lambda}} \int_0^{\lambda t} (V(s)-\theta)ds = \frac{1}{\sqrt{\lambda}}(X(\lambda t)-\theta\lambda t),\ \;\; t \geq 0,
\end{equation} 
which can also be expressed in the equivalent form 
\begin{equation}\label{Decomposition fvrm}
X_\lambda(t) = \frac{1}{\sqrt{\lambda}}\sum_{k=1}^{N(\lambda t)} (\hat{V}_{k-1}-\theta)(S_{k} -S_{k-1})+\frac{1}{\sqrt{\lambda}}(\hat{V}_{N(\lambda t)}-\theta)(\lambda t- S_{N(\lambda t)}),\ \;\; t \geq 0.
\end{equation}

The underlying idea of the normalization introduced in equation \eqref{normalized semi-markov} is to ensure that in \eqref{Decomposition fvrm} appears a sum of random variables having null mean (assuming that $\hat{V}_{0}$ has distribution $\pi$) with the classical central limit type scaling. From a physical point of view the above reparameterization guarantees that the number of changes of direction grow at the order of the square-root of the velocity. Indeed, as equation \eqref{scaled counting process} shows, the number of total changes of velocity in a unit time interval increases as $\lambda$. On the other hand, equation \eqref{normalized semi-markov} provides that the normalized velocity grows as $\sqrt{\lambda}$. This is the standard equilibrium condition allowing the convergence of the telegraph equation to the heat equation. Heuristically, this condition has the same meaning to the one that arises in the derivation of Brownian motion as limit of a random walk. Indeed, to obtain Brownian motion, a particle undergoing a random walk must perform a number of displacements that decay at the same order as the square root of their lengths (see \cite{Zauderer} sections 1.1 and 1.2). 

Having introduced the proper normalization, our task is to show the weak convergence of $X_\lambda$ as $\lambda \to +\infty$ to a scaled Brownian motion in the space $C[0,+\infty)$.

In equation \eqref{Decomposition fvrm} we can see that $X_\lambda$ can be decomposed into the sum of a random walk with a random number of summands and an interpolating term, which is proportional to the residual life of the counting process $N_\lambda$. The increase in the number of changes in direction results in arrival times occurring more frequently, which informally means that every instant becomes an arrival time. This behavior is formally proved by Theorem \ref{Residual life}, which shows that, 
\begin{equation*}
\sqrt{\lambda}\bigg(\frac{S_{N(\lambda t)}}{\lambda}-t\bigg) = \frac{S_{N(\lambda t)} -\lambda t}{\sqrt{\lambda}} \Rightarrow 0.
\end{equation*}
In worlds, the difference between the current time $t$ and the instant of the normalized last arrival time $\lambda^{-1}S_{N(\lambda t)}$ approaches to zero faster than $\lambda^{-1/2}$. This result allows us to study the asymptotic law of $X_\lambda(t)$ through the law of $X_\lambda(\lambda^{-1}S_{N(\lambda t)})$, which is a random walk where random number of steps is enumerated by the counting process $N(\lambda t)$. Hence, we define the random function on $D[0,+\infty)$, $Y_\lambda = (Y_\lambda(t))_{t\geq 0}$
\begin{equation}
Y_\lambda(t) =X_\lambda(\lambda^{-1}S_{N(\lambda t)}) =\frac{1}{\sqrt{\lambda}}\sum_{k=1}^{N(\lambda t)} (\hat{V}_{k-1}-\theta)(S_{k} -S_{k-1}),\ \;\; t \geq 0.
\end{equation}
From this consideration we readily obtain the following theorem.
\begin{theorem}\label{th res goes to zero}
Let $(\hat{V},S)$ be a Markov renewal process on a finite state space $\mathcal{V}$ and assume that the Markov chain $\hat{V}$ is irreducible and that $\mathsf{E}[S^2_1|\hat{V}_0 = v,V_1 = w]<+\infty$ for every $v,w \in \mathcal{V}$. Then
\begin{equation}
\lim_{\lambda \to +\infty}\mathsf{P}\{\sup_{t \in [0,T]}|X_\lambda(t)-Y_\lambda(t)| \geq \epsilon \} = 0,\ \;\; \epsilon >0.
\end{equation}
\end{theorem}
%\begin{proof}By looking at equation \eqref{Decomposition fvrm}, the finiteness of the state space $\mathcal{V}$ implies that it is sufficient to prove\begin{equation*}\sup_{t \in [0,T]}\frac{\lambda t-S_{N(\lambda t)}}{\sqrt{\lambda}} \Rightarrow 0.\end{equation*}But this follows immediately by means of Theorem \ref{Residual life}.\end{proof}
In light of the interpretation of the random variables $(\hat{V}_{k-1},\xi_k)$, the product $\eta_k = (\hat{V}_{k-1}-\theta)\xi_k$, $k \in \mathbb{N}$ describes the length of the normalized $k$-th displacement. The random variables $\eta_k$ have null mean and variance given by
\begin{align}
&\mathsf{E}_\pi[\eta^2_k] = \sum_{v\in \mathcal{V}}\int_0^{+\infty}(v-\theta)^2t^2\pi_vQ_{v\cdot}(dt)
\end{align}
while the auto-covariance is
\begin{equation}\label{autocovariance}
\mathsf{E}_\pi[\eta_1\eta_{1+k}] = \int_{0}^{+\infty}\int_{0}^{+\infty}\sum_{v\in \mathcal{V}}\sum_{w\in \mathcal{V}}(v-\theta)(w-\theta)st\pi_{v}Q_{v\cdot}(ds)p^{(k)}_{s;vw}Q_{w\cdot}(dt),\ \;\; k \in \mathbb{N}.
\end{equation}
From Theorem \ref{ergodic theorem Mr} and the fact that $\mathsf{E}_{\pi}[\eta_k] = 0$, we have the following estimate of the covariances of $\eta_k$
\begin{align}\label{estimate covariance}
|\mathsf{E}_\pi[\eta_1\eta_{1+k}]| &= \biggl|\int_{0}^{+\infty}\int_{0}^{+\infty}\sum_{v\in \mathcal{V}}\sum_{w\in \mathcal{V}}(v-\theta)(w-\theta)st\pi_{v}Q_{v\cdot}(ds)Q_{w\cdot}(dt)(p^{(k)}_{s;vw}-\pi_v)\biggr| \nonumber\\
&\leq \frac{K}{\min_{v'\in \mathcal{V}}\pi_v'}\rho^{k-1}\int_{0}^{+\infty}\int_{0}^{+\infty}\sum_{v\in \mathcal{V}}\sum_{w\in \mathcal{V}}|(v-\theta)(w-\theta)|st\pi_{v}\pi_wQ_{v\cdot}(ds)Q_{w\cdot}(dt)\nonumber\\
&= \frac{K}{\min_{v'\in \mathcal{V}}\pi_v'}\rho^{k-1}\biggr(\mathsf{E}_\pi[|\hat{V}_0-\theta|\cdot\xi_1]\biggl)^2,\ \;\;k \in \mathbb{N}.
\end{align}
In Lemma \ref{Fundamental lemma}, we have seen that, if the number of summands asymptotically behaves as a deterministic sequence, the random walk with a random number of summands converges weakly to the same limit of the same random walk with a deterministic number of terms. Hence we define the random function $\hat{Y}_\lambda = (\hat{Y}_\lambda(t), t \in [0,T])$ on $D[0,T]$, for some $T>0$,  
\begin{equation}\label{X hat}
\hat{Y}_\lambda(t) = \frac{1}{\sqrt{\lambda}}\sum_{k=1}^{\lfloor \lambda t \rfloor}\eta_k,\ \;\; t \in [0,T].
\end{equation}
The stationarity of $\{\eta_k\}_{k\geq 1}$ implies that the variance of $Y_\lambda$ becomes
\begin{align}\label{var X hat}
\mathsf{E}_\pi[\hat{Y}^2_\lambda(t)] = \frac{\lfloor \lambda t \rfloor}{\lambda}\mathsf{E}_\pi[(\hat{V}_0-\theta)^2\xi_1^2]+2\frac{\lfloor \lambda t \rfloor}{\lambda}\sum_{k=1}^{\lfloor \lambda t \rfloor-1}\biggl(1-\frac{k}{\lfloor \lambda t \rfloor}\biggr)\mathsf{E}_\pi[(\hat{V}_0-\theta)\xi_1(\hat{V}_1-\theta)\xi_{1+k}].
\end{align}
As a consequence of \eqref{estimate covariance}, the Cesaro sum contained in \eqref{var X hat} converges and then 
\begin{equation}
\lim_{\lambda \to +\infty}\mathsf{E}_\pi[\hat{Y}^2_\lambda(t)] = t\gamma^2
\end{equation}
where
\begin{align}\label{asympt var of X hat}
\gamma^2& = \sum_{v\in \mathcal{V}}\int_0^{+\infty}(v-\theta)^2s^2\pi_vQ_{v\cdot}(ds)\nonumber\\
&\ \ +2\sum_{k\geq 1}\sum_{v\in \mathcal{V}}\sum_{w\in \mathcal{V}}\int_{0}^{+\infty}\int_{0}^{+\infty}(v-\theta)(w-\theta)st\pi_{v}Q_{v\cdot}(ds)Q_{w\cdot}(dt)p^{(k)}_{s;vw}.
\end{align}
In the next Theorem we show the weak convergence of $\hat{Y}_\lambda$ to a scaled Brownian motion and we derive an alternative expression of the limiting variance.
\begin{theorem}\label{Th X hat}
Let $(\hat{V},S)$ be a Markov renewal process on a finite state space $\mathcal{V}$ and assume that the Markov chain $\hat{V}$ is irreducible and aperiodic with invariant distribution $\pi$ and that $\mathsf{E}_\pi[S^2_1]<+\infty$. Then, the process $\hat{Y}_\lambda$ defined in \eqref{X hat} satisfies $\hat{Y}_\lambda \Rightarrow \gamma W$ in $D[0,T]$, where $\gamma$ is expressed in \eqref{asympt var of X hat}. Moreover, for every $v_0 \in \mathcal{V}$, it holds that
\begin{align}\label{asympt var of X hat 2}
&\gamma^2 =\pi_{v_0}\mathsf{E}\biggl[\biggl(\sum_{k=1}^{\tau_1}\eta_k\biggr)^{2}\bigg|\hat{V}_0 = v_0\biggr]
\end{align}
where $\tau_1 = \inf\{k\geq1: \hat{V}_k = v_0\}$.
\end{theorem}
\begin{proof}
By Theorem \ref{phi mixing} the stationary sequence of random variables $\{\eta_k\}_{k\geq 1}$ is $\varphi$-mixing with parameter $\varphi_n = K\rho^{n-1}$. Therefore, we can apply Theorem 19.2 of \cite{Billingsley} to conclude that $\hat{Y}_\lambda \Rightarrow \gamma W$ in the space $D[0,T]$, where $\gamma^2 = \mathsf{E}_\pi[\eta^2_1]+2\sum_{k\geq1}\mathsf{E}_\pi[\eta_1\eta_{k+1}]$, which coincides with \eqref{asympt var of X hat}.

In order to prove the second statement of the theorem, we use the regenerative property of $\{\hat{V}_{k-1},\xi_k\}_{k\geq 1}$. Fix a velocity $v_0 \in \mathcal{V}$, by Theorem \ref{Regenerative property}, we have that $\{\eta_k\}_{k\geq 1}$ forms a delayed regenerative process with respect to the sequence $\{\tau_m\}_{m\geq 0}$ defined in \eqref{successive passage times}. The irreducibility of the chain and the finiteness of the state space $\mathcal{V}$ implies through \eqref{finite moments} that $\mathsf{E}[\tau^2_m]<+\infty$ and $\mathsf{E}[\tau_m-\tau_{m-1}] = \mathsf{E}[\tau_1|\hat{V}_0 = v_0] = \pi_{v_0}^{-1}$ for every $m \in \mathbb{N}$. Moreover, from Proposition \ref{prop 1}, it holds that
%for any $m \in \mathbb{N}$ and $A \in \mathcal{B}(\mathbb{R}^d)$,\begin{align*}&\mathsf{E}[\mathsf{1}_A(\eta_{\tau_{m-1}+1},...,\eta_{\tau_{m}})\mathsf{1}_{\{d\}}(\tau_{m}-\tau_{m-1})]=\mathsf{E}[\mathsf{1}_A(\eta_{\tau_{m-1}+1},...,\eta_{\tau_{m}})\mathsf{1}_{\{d\}}(\tau_{m}-\tau_{m-1})|\mathcal{F}_{\tau_{m-1}}]\\&= \mathsf{E}[\mathsf{1}_A(\eta_{1},...,\eta_{\tau_{1}})\mathsf{1}_{\{d\}}(\tau_{1})|\hat{V}_0 = v_0].\end{align*}hence, 
\begin{equation*}
\mathsf{E}\biggl[\sum_{k=1}^{\tau_1}\eta_k\bigg|\hat{V}_0 = v_0\biggr] = \pi_{v_0}^{-1}\mathsf{E}_\pi[\eta_1]=0. %\mathsf{E}[\tau_1|\hat{V}_0 = v_0]\mathsf{E}_\pi[\eta_1] = 0.
\end{equation*}
Now, \eqref{delayed reg2} and \eqref{assumption reg} entail that, for every $m \in \mathbb{N}$, 
\begin{align*}
\mathsf{E}\biggl[\biggl(\sum_{k=\tau_{m-1}+1}^{\tau_{m}}|\eta_k|\biggr)^2\biggr]& = \mathsf{E}\biggl[\biggl(\sum_{k=1}^{\tau_{1}}|\eta_k|\biggr)^2\bigg|\hat{V}_0 = v_0\biggr]\\
&\leq \max_{v \in \mathcal{V}}(v-\theta)^2\;\mathsf{E}\biggl[\biggl(\sum_{k=1}^{\tau_{1}}\xi_k\biggr)^2\bigg|\hat{V}_0 = v_0\biggr] \\
&\leq \max_{v \in \mathcal{V}}(v-\theta)^2\max_{v,w \in \mathcal{V}}\mathsf{E}[\xi^2_1|\hat{V}_1 = w,\hat{V}_0 = v]\mathsf{E}[\tau^2_1|\hat{V}_0 = v_0].
\end{align*}
Thus, $\mathsf{E}[(\sum_{k=\tau_{m-1}+1}^{\tau_{m}}|\eta_k|)^2]<+\infty$ for every $m \in \mathbb{N}_0$, from which we can apply Theorem \ref{Clt regenerative processes}, and obtain
\begin{equation*}
\hat{Y}_{\lambda} \Rightarrow \gamma W
\end{equation*}
in $D[0,T]$, where
\begin{equation*}
\gamma^2 = \pi_{v_0}\mathsf{E}\biggl[\biggl(\sum_{k=1}^{\tau_1}\eta_k\biggr)^{2}\bigg|\hat{V}_0 = v_0\biggr].
\end{equation*}
\end{proof}
We are now in the position to prove the main theorem of this paper.
\begin{theorem}\label{main Theorem}
Let $V = (V(t))_{t \geq 0}$ be a semi-Markov process with respect to $(\hat{V},S)$, a Markov renewal process on a finite state space $\mathcal{V}$, where the embedded Markov chain $\hat{V}$ is irreducible and aperiodic with invariant distribution $\pi$. Let us put $\mu=\mathsf{E}_\pi[S_1]$ and assume that $\mathsf{E}_\pi[S^2_1]<+\infty$. Then, the process $X_\lambda = (X_\lambda(t))_{t \geq 0}$ defined by
\begin{equation*}
X_\lambda(t) =\frac{1}{\sqrt{\lambda}} \int_0^{\lambda t}(V(s)-\theta)ds,
\end{equation*}
where $\theta =  \mu^{-1}\mathsf{E}_\pi[\hat{V}_0S_1]$, satisfies the following weak limit in $C[0,+\infty)$ 
\begin{equation}
X_\lambda \Rightarrow \mu^{-1/2}\gamma W
\end{equation}
where $\gamma$ is provided by \eqref{asympt var of X hat} or \eqref{asympt var of X hat 2}.
\end{theorem}
\begin{proof}
From Theorem \ref{prop 1} we have that $\mu = (\mathsf{E}[\tau_1|\hat{V}_0 = v_0])^{-1}\mathsf{E}[\sum_{k=1}^{\tau_1}\xi_k|V_0 = v_0]$. Let us assume, without loss of generality, that $\mu > 1$. Then, by Theorem \ref{renewal theorem}, the random function of $D[0,T]$
\begin{equation}
\Phi_{\lambda}(t) = 
\begin{cases}
\frac{N(\lambda t)}{\lambda}, &\frac{N(\lambda T)}{\lambda} \leq T\\
\frac{t}{\mu}, &\frac{N(\lambda T)}{\lambda} > T
\end{cases}
\end{equation}
satisfies $\Phi_{\lambda}\Rightarrow_\lambda \phi$ on $D[0,T]$, where $\phi(t) = \mu^{-1} t$. By theorems \ref{renewal theorem}, \ref{Th X hat} and Lemma of page 151 of \cite{Billingsley} we obtain that $\hat{Y}_\lambda \circ \Phi_\lambda \Rightarrow \gamma W \circ \phi$. Moreover, $Y_\lambda = \hat{Y}_\lambda \circ \Phi_\lambda$ on the set $\{\lambda^{-1}N(\lambda T) \leq T\}$, the probability of which goes to one by again Theorem \ref{renewal theorem} and the fact that $\mu >1$. In conclusion, $Y_\lambda \Rightarrow \gamma W \circ \phi$ in $D[0,T]$, where $\gamma W \circ \phi$ is a Gaussian process with the same distribution of $\mu^{-1/2}\gamma W$. By applying Theorem \ref{th res goes to zero}, we can state that $X_\lambda \Rightarrow \mu^{-1/2}\gamma W$ in $D[0,T]$.  Now, from $\mathsf{P}\{X_\lambda \in C[0,T]\} \equiv \mathsf{P}\{\mu^{-1/2}\gamma W \in C[0,T]\} = 1$, by example 2.9 of \cite{Billingsley}, $X_\lambda \Rightarrow \mu^{-1/2}\gamma W$ in $C[0,T]$. From this, we can finally obtain that $X_\lambda \Rightarrow\mu^{-1/2}\gamma W$ in $C[0,+\infty)$.
\end{proof}
As a simple corollary of the previous theorem (easily proved by means of the continuous mapping theorem) we have a weak law of large numbers (in a functional settings) for the process $X(\lambda t)$.
\begin{corollary}\label{weak law integral semi-Markov}
Let $V = (V(t))_{t \geq 0}$ be a semi-Markov process with respect to $(\hat{V},S)$, a Markov renewal process on a finite state space $\mathcal{V}$, where the embedded Markov chain $\hat{V}$ is irreducible and aperiodic with invariant distribution $\pi$. Let us put $\mu=\mathsf{E}_\pi[S_1]$ and assume that $\mathsf{E}_\pi[S^2_1]<+\infty$. Then, $X = (X(t))_{t \geq 0}$ the integral of $V$,
\begin{equation*}
X(t) = \int_0^tV(s)ds
\end{equation*}
satisfies the following limit
\begin{equation}
\lim_{\lambda \to +\infty}\mathsf{P}\bigg\{\sup_{t \in [0,T]}\bigg|\frac{X(\lambda t)}{\lambda} - \frac{\mathsf{E}_\pi[\hat{V}_0S_1]}{\mathsf{E}_\pi[S_1]}\bigg|>\epsilon\bigg\} = 0,\ \;\; \epsilon>0.
\end{equation}
\end{corollary}
\begin{remark}\label{remark deriva}
Under the assumptions of Theorem \ref{main Theorem} we consider an alternative normalization of the integral of a semi-Markov process which introduces a drift. Let $\bar{X}_\lambda = (\bar{X}_\lambda(t))_{t\geq 0}$ the random function of $C[0+\infty)$ defined as the integral of $\bar{V}_\lambda = (\bar{V}_\lambda(t))_{t\geq 0}$, $\bar{V}_\lambda(t) = \mathsf{E}[V(\lambda t)] + \sqrt{\lambda}(V(\lambda t)-\theta)$,
\begin{equation}
\bar{X}_\lambda(t) = \frac{1}{\lambda}\int_0^{\lambda t}\mathsf{E}[V(s)]ds + \frac{1}{\sqrt{\lambda}}\int_0^{\lambda t}(V(s)-\theta)ds =  \lambda^{-1}\mathsf{E}[X(\lambda t)] + \lambda^{-1/2}(X(\lambda t)-\theta \lambda t). 
\end{equation}
Corollary \ref{weak law integral semi-Markov} ensures that, for every $t \in [0,T]$, 
\begin{equation*}
\lambda^{-1}X(\lambda t) = \frac{1}{\lambda}\int_0^{\lambda t}V(s)ds\Rightarrow \theta t,
\end{equation*}
moreover, the uniform integrability of the sequence $\lambda^{-1}X(\lambda t)$ (which follows from the boundedness of $V$), entails
\begin{equation}\label{uniform int}
\lim_{\lambda \to +\infty}\lambda^{-1}\mathsf{E}[X(\lambda t)] = \lim_{\lambda \to +\infty}\frac{1}{\lambda}\int_0^{\lambda t}\mathsf{E}[V(s)]ds =\theta t
\end{equation}
for every $t \in [0,T]$. As a result of \eqref{uniform int}, Theorem \ref{main Theorem} implies that
\begin{equation}
\bar{X}_\lambda \Rightarrow \text{i}\theta + \mu^{-1/2}\gamma W
\end{equation}
where $\text{i}$ is the identity function, $\mu = \mathsf{E}_\pi[S_1]$ and $\gamma$ is given in \eqref{asympt var of X hat} or \eqref{asympt var of X hat 2}. 
We remark that \eqref{uniform int} can be also proved by means of \eqref{limit of the mean of semi-Markov} and the dominated convergence theorem (after the change of variable $s' = s/\lambda$).
\end{remark}
\section{Weak convergence of the integral of an alternating renewal process}\label{sezioneConvergenzaIntegralRinnovoAlternato}
Let $\mathcal{V} = \{v_1,...,v_m\}$ and define $v_{rm+i} \coloneqq v_i$ for $r \in \mathbb{N}_0$ and $i=1,...,m$. An alternating renewal process is a particular case of a semi-Markov process in which a particle deterministically moves from state $v_k$ to state $v_{k+1}$ and the cycle restarts after $m$ steps. This means that the Markov renewal-kernel is of the form
\begin{align}
&\mathsf{P}\{\hat{V}_{k+1} = v_{j},S_{k+1}-S_{k}\leq t|\hat{V}_{k} = v_i\}%=\begin{cases}0 &j \neq i+1 \\Q_{v_i v_{i+1}}(t) &j = i+1\end{cases} \;\; 
=Q_{v_iv_j}(t)=\mathsf{1}_{\{v_{i+1}\}}(v_j)F_{v_i v_{j}}(t)
\end{align}
for every $i,j=1,...,m$, $k\in \mathbb{N}_0$, and $t \geq 0$. From the definition, it follows that the embedded Markov chain $\hat{V}$ is periodic with period equals to $m$ and the transition probabilities are given by $p_{v_iv_j} = \mathsf{1}_{\{v_{i+1}\}}(v_j)$, while the invariant distribution is given by $\pi_{v_j} = m^{-1}$, $j=1,...,m$. Moreover it holds that
\begin{equation*}
\mathsf{P}\{S_{k+1}-S_{k}\leq t|\hat{V}_{k} = v_i,\hat{V}_{k+1} = v_{i+1}\} = F_{v_iv_{i+1}}(t) = \mathsf{P}\{S_{k+1}-S_{k}\leq t|\hat{V}_{k} = v_i\} = Q_{v_i\cdot}(t).
\end{equation*}
We note that, if $\mathsf{P}\{\hat{V}_0 = v_1\}=1$, then $\{\hat{V}_k\}_{k\geq 0}$ becomes a deterministic sequence. Indeed, $\mathsf{P}\{V_k = v_i\} = 1$ if there exists $r \in \mathbb{N}_0$ such that $k = rm+i-1$, $i=1,...,m$ and, by letting with $\xi_k = S_k - S_{k-1}$, the random variables forming the sequence $\{\xi_k\}_{k \geq 1}$ are independent and satisfies $\xi_{rm+i} =_d \xi_i$. 

The alternating renewal process $V = (V(t))_{t \geq 0}$ based on $\{\hat{V}_k,S_k\}_{k \geq 0}$ can be written down as
\begin{equation}V(t) = %\sum_{k\geq 0}V_{k-1}\mathsf{1}_{[S_{k-1},S_{k})}(t) = 
\sum_{k=1}^m\hat{V}_{k-1}\sum_{r\geq 0}\mathsf{1}_{\{rm+k-1\}}(N(t)),
\end{equation}
where $N$ is the counting process related to $S$ and defined in \eqref{counting process}.

%, and hence,\begin{equation}\mathsf{E}[V(t)] = \sum_{k=1}^m\sum_{r\geq 0}\mathsf{E}[\hat{V}_{k-1}\mathsf{1}_{\{rm+k-1\}}(N(t))]\end{equation}
%while\begin{equation}\mathsf{E}[V(t)|V(0) = v_1] = \sum_{i=1}^mv_i\sum_{r\geq 0}\mathsf{P}\{N(t) = rm+i-1|\hat{V}_0 = v_1\}.\end{equation}
Let us denote by $\mu_i = \mathsf{E}[\xi_1|\hat{V}_0 = v_i]$ and $\sigma^2_i = \mathsf{V}[\xi_1|\hat{V}_0 = v_i]$. Then,
\begin{equation}\label{theta alternating renewal process}
\theta = \frac{\mathsf{E}_\pi[\hat{V}_0S_1]}{\mathsf{E}_\pi[S_1]}=\frac{\sum_{i=1}^mv_i\mu_i}{\sum_{i=1}^m\mu_i},
\end{equation}
and in this case the integral of $V_{\lambda}(t) = \sqrt{\lambda}(V(\lambda t)-\theta)$ can be written down as
\begin{align}
&X_{\lambda}(t) = \sum_{k=1}^m(\hat{V}_{k-1}-\theta)\int_0^t\sum_{r\geq 0}\mathsf{1}_{\{rm+k-1\}}(N(\lambda s))ds
\end{align}
or also, as it is done in \eqref{Decomposition fvrm},
\begin{align}
&X_{\lambda}(t) = \frac{1}{\sqrt{\lambda}}\sum_{k = 1}^{N(\lambda t)}(\hat{V}_{k-1}-\theta)\xi_k + \frac{1}{\sqrt{\lambda}}(\hat{V}_{N(\lambda t)}-\theta)\biggl(\lambda t-\sum_{k = 1}^{N(\lambda t)}\xi_k\biggr).
\end{align}
Let us denote with \begin{align}\label{formulas}
&\mu = \frac{1}{m}\sum_{i=1}^m\mu_i,&\gamma^2 =\frac{1}{m}\sum_{i=1}^m\sigma^2_i\biggl(v_i-\frac{\sum_{i=1}^mv_i\mu_i}{\sum_{i=1}^m\mu_i}\biggr)^2.
\end{align}
Since the alternating renewal process is periodic, we cannot direct apply Theorem \ref{main Theorem}. The next statement provides the conditions for the weak convergence in this new framework.
\begin{theorem}
Let $V = (V(t))_{t \geq 0}$ be an alternating renewal process with state space $\mathcal{V}$ and assume that $\mathsf{E}[S_1^2|\hat{V}_0 = v]<+\infty$ for every $v \in \mathcal{V}$. Then, the process $X_\lambda = (X_\lambda(t))_{t \geq 0}$ defined by
\begin{equation}
X_\lambda(t) =\frac{1}{\sqrt{\lambda}} \int_0^{\lambda t}(V(s)-\theta)ds,
\end{equation}
where $\theta$ is defined in \eqref{theta alternating renewal process},
satisfies the following weak limit in $C[0,+\infty)$  
\begin{equation}
X_\lambda \Rightarrow\mu^{-1/2}\gamma W
\end{equation}
where $\mu$ and $\gamma$ are provided in \eqref{formulas}.
\end{theorem}
\begin{proof}
From Theorem \ref{Regenerative property}, the sequence $\{(V_{k-1},\xi_k)\}_{k\geq 1}$ forms a delayed regenerative process with regeneration epochs $\tau_n = \inf\{k>\tau_{n-1}: V_k = v_1\}$, which satisfy $\tau_{n+1}-\tau_{n} = m$ for $n \in \mathbb{N}$. By \eqref{delayed reg2}, for every $n \in \mathbb{N}$
\begin{align*}
\frac{\mathsf{E}\biggl[\sum_{k=\tau_{n}+1}^{\tau_{n+1}}V_{k-1}\xi_k\biggr]}{\mathsf{E}[\tau_{n+1}-\tau_n]} = \frac{1}{m}\mathsf{E}\biggl[\sum_{k=1}^{\tau_1}V_{k-1}\xi_k\bigg|V_0 = v_1\biggr]=\frac{1}{m}\sum_{k=1}^{m}v_k\mathsf{E}[\xi_k|V_0 = v_1] =\frac{1}{m}\sum_{i=1}^{m}v_i\mu_i,
\end{align*}
from which we have
\begin{align}\label{mu 0}
\mathsf{E}\biggl[\sum_{k=\tau_{1}+1}^{\tau_{2}}(V_{k-1}-\theta)\xi_k\biggr] = \sum_{i=1}^{m}v_i\mu_i-\frac{\sum_{i=1}^mv_i\mu_i}{\sum_{i=1}^m\mu_i}\sum_{i=1}^m\mu_i = 0.
\end{align}
Moreover,
\begin{align*}
\mathsf{E}&\biggl[\biggl(\sum_{k=\tau_{n}+1}^{\tau_{n+1}}|V_{k-1}-\theta|\xi_k\biggr)^2\biggr] = \mathsf{E}\biggl[\biggl(\sum_{k=1}^{\tau_{1}}|V_{k-1}-\theta|\xi_k\biggr)^2\bigg|\hat{V}_0 = v_1\biggr] \\&\leq\max_{i=1,...,m}(v_i-\theta)^2 \mathsf{E}\biggl[\biggl(\sum_{k=1}^{m}\xi_k\biggl)^2\bigg|V_0 = v_1\biggr] = \max_{i=1,...,m}(v_i-\theta)^2\bigg[\sum_{i=1}^{m}\sigma^2_i+\bigg(\sum_{i=1}^{m}\mu_i\bigg)^2\bigg] <+\infty
\end{align*}
Therefore, from Theorem \ref{Clt regenerative processes} it follows that the process $\hat{Y}_\lambda$ defined as
\begin{equation*}
\hat{Y}_\lambda(t) = \frac{1}{\sqrt{\lambda}}\sum_{k=1}^{\lfloor \lambda t\rfloor}(V_{k-1}-\theta)\xi_k,\ \;\; t \in [0,T]
\end{equation*}
satisfies $\hat{Y}_\lambda \Rightarrow \gamma W$ on $D[0,T]$, where, from \eqref{mu sigma reg}, and by keeping in mind \eqref{mu 0},
\begin{align*}
\gamma^2 &= \big(\mathsf{E}[\tau_2-\tau_1]\big)^{-1}\mathsf{E}\biggl[\biggl(\sum_{k=\tau_{1}+1}^{\tau_{2}}(V_{k-1}-\theta)\xi_k\biggr)^2\biggr] = \frac{1}{m}\mathsf{E}\biggl[\biggl(\sum_{k=1}^{\tau_{1}}(V_{k-1}-\theta)\xi_k\biggr)^2\bigg|\hat{V}_0 = v_1\biggr]\\
&= \frac{1}{m}\sum_{i=1}^m\sigma^2_i\biggl(v_i-\frac{\sum_{i=1}^mv_i\mu_i}{\sum_{i=1}^m\mu_i}\biggr)^2.
\end{align*}
Now, Theorem \ref{renewal theorem} implies that 
\begin{equation*}
\sup_{t \in [0,T]}\biggl|\frac{N(\lambda t)}{\lambda}-\frac{tm}{\sum_{i=1}^m\mu_i}\biggr| \Rightarrow 0.
\end{equation*}
Thus, we can use Lemma \ref{Fundamental lemma} together with Theorem \ref{Residual life}, to conclude that
\begin{equation*}
X_\lambda  \Rightarrow \mu^{-1/2}\gamma W
\end{equation*}
on $D[0,T]$. By using the same arguments of proof of Theorem \ref{main Theorem} we obtain the thesis.
\end{proof}
\begin{remark}\label{remark deriva alternatingrenewalprocess}
If, according to Remark \ref{remark deriva}, we consider the alternative normalization 
\begin{align}
\bar{X}_\lambda(t) &= \frac{1}{\lambda}\int_0^{\lambda t}\mathsf{E}[V(s)]ds + \frac{1}{\sqrt{\lambda}}\int_0^{\lambda t}(V(s)-\theta)ds=\lambda^{-1}\mathsf{E}[X(\lambda t)] + \lambda^{-1/2}(X(\lambda t)-\theta \lambda t)\nonumber\\ 
&=\sum_{k=1}^m\int_0^t\sum_{r\geq 0}\mathsf{E}[\hat{V}_{k-1}\mathsf{1}_{\{rm+k-1\}}(N(\lambda s))]ds+\sum_{k=1}^m(\hat{V}_{k-1}-\theta)\int_0^t\sum_{r\geq 0}\mathsf{1}_{\{rm+k-1\}}(N(\lambda s))ds.
\end{align}
The previous theorem ensures that, for every $t \in [0,T]$, 
\begin{equation*}
\bar{X}_\lambda(t) -\lambda^{-1}\mathsf{E}[X(\lambda t)] = \frac{1}{\sqrt{\lambda}}\int_0^{\lambda t}(V(s)-\theta)ds \Rightarrow \mu^{-1/2}\gamma W(t)
\end{equation*}
which implies
\begin{equation*}
\frac{X(\lambda t)}{\lambda}=\frac{1}{\lambda}\int_0^{\lambda t}V(s)ds\Rightarrow \theta t.
\end{equation*}
Then, by uniform integrability we obtain
\begin{equation*}\lim_{\lambda \to +\infty}\frac{1}{\lambda}\mathsf{E}[X(\lambda t)] = \theta t
\end{equation*}
and hence
\begin{equation}
\bar{X}_\lambda \Rightarrow \text{i}\theta + \mu^{-1/2}\gamma W
\end{equation}
where $\text{i}$ is the identity function.
\end{remark}
\subsection{Application to the generalized telegraph process}\label{sezioneApplicazioni}
Here we show an application of the previous results to the asymmetric telegraph process (see \cite{BNO2001, C2022, SZ2004}). 
\\Let $(\hat{V},S)$ be an alternating renewal process where $\mathcal{V} = \{v_1,v_2\}$, with $v_1,v_2 \in \mathbb{R}$, and  
\begin{equation}\label{alternating poisson process}
Q_{v_iv_{j}}(t) = \mathsf{1}_{\{v_{i+1}\}}(v_j)(1-e^{-\lambda_i t}).
\end{equation}
Hence,
\begin{align*}
&\mathsf{P}\{V_k = v_i,S_k-S_{k-1}\leq t|V_{k-1} = v_{i-1}\} = \mathsf{P}\{S_k-S_{k-1}\leq t|V_{k-1} = v_{i-1}\} = \begin{cases}
1-e^{-\lambda_1t}, &v_{i-1}=v_1\\
1-e^{-\lambda_2t}, &v_{i-1}=v_2.
\end{cases}\end{align*}
%\begin{align*}&\mathsf{P}\{S_k-S_{k-1}\leq t,V_k = v_2|V_{k-1} = v_1\} = \mathsf{P}\{S_k-S_{k-1}\leq t|V_{k-1} = v_1\} = 1-e^{-\lambda_1t} \\&\mathsf{P}\{S_k-S_{k-1}\leq t,V_k = v_1|V_{k-1} = v_2\} = \mathsf{P}\{S_k-S_{k-1}\leq t|V_{k-1} = v_2\}=1-e^{-\lambda_2t}\end{align*}
If we set $N(t) = \max\{k\in \mathbb{N}_0: S_k \leq t\}$, then $N = (N(t))_{t\geq0}$ is the alternating Poisson process. It can be proved that
\begin{align*}&\mathsf{P}\{N(t) = n|V(0) = v_1\} = \begin{cases}
(\lambda_1t)^{k}(\lambda_2t)^{k}e^{-\lambda_1 t}W_{k,k+1}\bigl(t(\lambda_1-\lambda_2)\bigr), &n=2k\\
(\lambda_1t)^{k+1}(\lambda_2t)^{k}e^{-\lambda_1 t}W_{k+1,k+1}\bigl(t(\lambda_1-\lambda_2)\bigr), &n=2k+1
\end{cases}\end{align*}
%\begin{align*}&\mathsf{P}\{N(t) = 2k|V(0) = v_1\} = (\lambda_1t)^{k}(\lambda_2t)^{k}e^{-\lambda_1 t}W_{k,k+1}\bigl(t(\lambda_1-\lambda_2)\bigr)\\&\mathsf{P}\{N(t) = 2k+1|V(0) = v_1\} = (\lambda_1t)^{k+1}(\lambda_2t)^{k}e^{-\lambda_1 t}W_{k+1,k+1}\bigl(t(\lambda_1-\lambda_2)\bigr)\end{align*}
where $W_{\alpha,\beta}(x) = (\Gamma(\alpha)\Gamma(\beta))^{-1}\int_{0}^1t^{\alpha-1}(1-t)^{\beta-1}e^{-x t}dt$.

Let $V(t) = V_{N(t)}$, from \eqref{alternating poisson process}, it follows that $V$ is a Markov process with generator 
\begin{equation*}
A=\begin{pmatrix}
-\lambda_1 & \lambda_1 \\
\lambda_2 & -\lambda_2 
\end{pmatrix},
\end{equation*}
then, Kolmogorov's forward equation can be used to obtain that
%\begin{align}&\mathsf{P}\{V(t) = v_1|V(0) = v_1\} =\frac{\lambda_2}{\lambda_1+\lambda_2}+\frac{\lambda_1}{\lambda_1+\lambda_2}e^{-(\lambda_1+\lambda_2)t}\\&\mathsf{P}\{V(t) = v_2|V(0) = v_1\}= \frac{\lambda_1}{\lambda_1+\lambda_2}-\frac{\lambda_1}{\lambda_1+\lambda_2}e^{-(\lambda_1+\lambda_2)t}\end{align}
\begin{align}&\mathsf{P}\{V(t) = v|V(0) = v_1\} =\begin{cases}
\frac{\lambda_2}{\lambda_1+\lambda_2}+\frac{\lambda_1}{\lambda_1+\lambda_2}e^{-(\lambda_1+\lambda_2)t}, &v=v_1\\
\frac{\lambda_1}{\lambda_1+\lambda_2}-\frac{\lambda_1}{\lambda_1+\lambda_2}e^{-(\lambda_1+\lambda_2)t}, &v=v_2
\end{cases}\end{align}
and
%\begin{align}&\mathsf{P}\{V(t) = v_1|V(0) = v_2\} =\frac{\lambda_2}{\lambda_1+\lambda_2}-\frac{\lambda_2}{\lambda_1+\lambda_2}e^{-(\lambda_1+\lambda_2)t}\\&\mathsf{P}\{V(t) = v_2|V(0) = v_2\}= \frac{\lambda_1}{\lambda_1+\lambda_2}+\frac{\lambda_2}{\lambda_1+\lambda_2}e^{-(\lambda_1+\lambda_2)t}\end{align}
\begin{align}&\mathsf{P}\{V(t) = v|V(0) = v_2\} =\begin{cases}
\frac{\lambda_2}{\lambda_1+\lambda_2}-\frac{\lambda_2}{\lambda_1+\lambda_2}e^{-(\lambda_1+\lambda_2)t}, &v=v_1\\
\frac{\lambda_1}{\lambda_1+\lambda_2}+\frac{\lambda_2}{\lambda_1+\lambda_2}e^{-(\lambda_1+\lambda_2)t}, &v=v_2.
\end{cases}\end{align}
If $(p,1-p)$ is the initial distribution of $V(0)$, then the law of $V$ takes the form
%\begin{align}&\mathsf{P}\{V(t)=v_1\} = \frac{\lambda_2}{\lambda_1+\lambda_2}+\frac{p_1\lambda_1-p_2\lambda_2}{\lambda_1+\lambda_2}e^{-(\lambda_1+\lambda_2)t} \\&\mathsf{P}\{V(t)=v_2\} = \frac{\lambda_1}{\lambda_1+\lambda_2}-\frac{p_1\lambda_1-p_2\lambda_2}{\lambda_1+\lambda_2}e^{-(\lambda_1+\lambda_2)t}
%\end{align}
\begin{align}&\mathsf{P}\{V(t) = v\} =\begin{cases}
\frac{\lambda_2}{\lambda_1+\lambda_2}+\frac{p\lambda_1-(1-p)\lambda_2}{\lambda_1+\lambda_2}e^{-(\lambda_1+\lambda_2)t}, &v=v_1\\
\frac{\lambda_1}{\lambda_1+\lambda_2}-\frac{p\lambda_1-(1-p)\lambda_2}{\lambda_1+\lambda_2}e^{-(\lambda_1+\lambda_2)t}, &v=v_2
\end{cases}\end{align}
from which we obtain
\begin{equation}
\mathsf{E}[V(t)] = \frac{v_1\lambda_2+v_2\lambda_1}{\lambda_1+\lambda_2}+\frac{(p\lambda_1-(1-p)\lambda_2)(v_1-v_2)}{\lambda_1+\lambda_2}e^{-(\lambda_1+\lambda_2)t}
\end{equation}
and
\begin{equation}
\theta = \frac{\mathsf{E}_\pi[\xi_1V_0]}{\mathsf{E}_\pi[\xi_1]} = \frac{v_1\lambda_2+v_2\lambda_1}{\lambda_1+\lambda_2}.
\end{equation}
Let $X$ be the integral of $V$, then the density of $X(t)$
%\begin{align}\label{probTelegrafoGeneralizzato}&f(x)= \frac{ e^{-\frac{\lambda_1(x-v_2t)+\lambda_2(v_1t-x)}{v_1-v_2}}}{(v_1-v_2)}\Biggl[\frac{\lambda_1+\lambda_2}{2}I_0\Bigl(\frac{2\sqrt{\lambda_1\lambda_2}}{v_1-v_2}\sqrt{(x-v_2t)(v_1t-x)} \Bigr)\nonumber\\&+\frac{v_1-v_2}{2}\frac{\sqrt{\lambda_1\lambda_2}t}{\sqrt{(x-v_2t)(v_1t-x)}}I_1\Bigl(\frac{2\sqrt{\lambda_1\lambda_2}}{v_1-v_2}\sqrt{(x-v_2t)(v_1t-x)}\Bigr)\Biggr]\end{align}
satisfies the following hyperbolic equation
\begin{align}\label{hyperbolic pde}
&\frac{\partial^2u}{\partial t^2} + (\lambda_1+\lambda_2)\frac{\partial u}{\partial t} +(v_1+v_2)\frac{\partial^2u}{\partial x\partial t} +(v_1\lambda_2+v_2\lambda_1)\frac{\partial u}{\partial x}= -v_1v_2\frac{\partial^2u}{\partial x^2}.
\end{align}
In this case, $X$ can be interpreted as the trajectory of a particle moving with constant velocities $v_1,v_2$ and subject to reversals of direction at the jump times of the alternating Poisson process of rates $\lambda_1,\lambda_2$. Therefore the mean length of each displacement is $1/\lambda_1$ and $1/\lambda_2$ according to the current speed.

Then $\bar{X}_\lambda$, the integral of $V$ under the reparametrization of Remark \ref{remark deriva alternatingrenewalprocess}, becomes
\begin{align}
&\bar{X}_\lambda(t) =\frac{(p\lambda_1-(1-p)\lambda_2)(v_1-v_2)}{\lambda(\lambda_1+\lambda_2)^2}(1-e^{-(\lambda_1+\lambda_2)\lambda t})\!+\!\int_0^{t}\sqrt{\lambda}\biggl(\!V(\lambda s)-\frac{v_1\lambda_2+v_2\lambda_1}{\lambda_1+\lambda_2}\bigl(1-\frac{1}{\sqrt{\lambda}}\bigr)\!\!\biggr)ds.
\end{align}
If we define the process $T_\lambda = (T_\lambda(t))_{t\geq0}$ by
\begin{equation}
T_\lambda(t) = \int_0^{t}\sqrt{\lambda}\biggl(V(\lambda s)-\frac{v_1\lambda_2+v_2\lambda_1}{\lambda_1+\lambda_2}\bigl(1-\frac{1}{\sqrt{\lambda}}\bigr)\biggr)ds,
\end{equation}
then $\bar{X}_\lambda -T_\lambda \Rightarrow 0$, and so $\bar{X}_\lambda$ and $T_\lambda$ have the same weak limit. Now, $T_\lambda$ is an asymmetric telegraph process with rates $\lambda\lambda_1$, $\lambda\lambda_2$ and velocities 
\begin{align*}
&\sqrt{\lambda}(v_1-\theta)+\theta = \frac{\sqrt{\lambda}\lambda_1(v_1-v_2)+v_1\lambda_2+v_2\lambda_1}{\lambda_1+\lambda_2}, &\sqrt{\lambda}(v_2-\theta)+\theta = \frac{\sqrt{\lambda}\lambda_2(v_2-v_1)+v_1\lambda_2+v_2\lambda_1}{\lambda_1+\lambda_2}.
\end{align*}
%Note that\begin{align*}&v_1-\theta = \frac{\lambda_1}{\lambda_1+\lambda_2}(v_1-v_2) &v_2-\theta = \frac{\lambda_2}{\lambda_1+\lambda_2}(v_2-v_1)\end{align*}
We note that, in the limit, the two velocities become of opposite signs. Thus, according to Remark \ref{remark deriva alternatingrenewalprocess}, we can conclude that an asymmetric telegraph process $T_\lambda$ with those parameters converges weakly to Brownian motion with drift and scaling parameter respectively given by
\begin{align*}
&\theta = \frac{v_1\lambda_2+v_2\lambda_1}{\lambda_1+\lambda_2} ,
&\frac{\gamma^2}{\mu} = %\frac{\frac{1}{\lambda_1^2}(v_1-\theta)^2+\frac{1}{\lambda_2^2}(v_2-\theta)^2}{\frac{1}{\lambda_1}+\frac{1}{\lambda_2}} = \frac{\frac{1}{\lambda_1^2}\bigl(\frac{\lambda_1(v_1-v_2)}{\lambda_1+\lambda_2}\bigr)^2+\frac{1}{\lambda_2^2}\bigl(\frac{\lambda_2(v_2-v_1)}{\lambda_1+\lambda_2}\bigr)^2}{\frac{1}{\lambda_1}+\frac{1}{\lambda_2}} =
2\frac{\lambda_1\lambda_2(v_1-v_2)^2}{(\lambda_1+\lambda_2)^{3}}.
\end{align*}
Moreover, from \eqref{hyperbolic pde} it follows that the density of $T_\lambda$ is a solution of the following partial differential equation
%\begin{align}\label{equazioneTelegrafoGeneralizzato}&\frac{\partial^2f_\lambda}{\partial t^2} + \lambda(\lambda_1+\lambda_2)\frac{\partial f_\lambda}{\partial t} +\sqrt{\lambda}(v_1+v_2)\frac{\partial^2f_\lambda}{\partial x\partial t}-\bigl(\lambda(v_1-\theta)(v_2-\theta)(1+\sqrt{\lambda}^{-1}\theta)+\theta^2\bigr)\frac{\partial^2f_\lambda}{\partial x^2} \nonumber\\&= \lambda(v_1\lambda_2+v_2\lambda_1)\frac{\partial f_\lambda}{\partial x}\end{align}which can be rewritten as
\begin{align}\label{equazioneTelegrafoGeneralizzato}
\frac{1}{\lambda(\lambda_1+\lambda_2)}&\frac{\partial^2u}{\partial t^2} + \frac{\partial u}{\partial t} +\biggl(\frac{2(v_1\lambda_2+v_2\lambda_1)}{\lambda(\lambda_1+\lambda_2)^2}+\frac{(v_2-v_1)(\lambda_1-\lambda_2)}{\sqrt{\lambda}(\lambda_1+\lambda_2)^2}\biggr)\frac{\partial^2u}{\partial x\partial t}+\frac{v_1\lambda_2+v_2\lambda_1}{\lambda_1+\lambda_2}\frac{\partial u}{\partial x}\nonumber\\
&=\biggl(\frac{\lambda_1\lambda_2(v_1-v_2)^2}{(\lambda_1+\lambda_2)^{3}}-\frac{(v_1\lambda_2+v_2\lambda_1)(\lambda_1-\lambda_2)(v_1-v_2)}{\sqrt{\lambda}(\lambda_1+\lambda_2)^2}-\frac{(v_1\lambda_2+v_2\lambda_1)^{2}}{\lambda(\lambda_1+\lambda_2)^{3}}\biggr)\frac{\partial^2 u}{\partial x^2},
\end{align}
and, as $\lambda\to+\infty$, we obtain 
\begin{equation}
\frac{\partial u}{\partial t}+\frac{v_1\lambda_2+v_2\lambda_1}{\lambda_1+\lambda_2}\frac{\partial u}{\partial x}=\frac{\lambda_1\lambda_2(v_1-v_2)^2}{(\lambda_1+\lambda_2)^{3}}\frac{\partial^2u}{\partial x^2},
\end{equation}
which is the equation governing the law of Brownian motion with drift $\theta$ and scaling parameter $\mu^{-1/2}\gamma$. 

In the case in which $v_1=-v_2 = v_0$ and $\lambda_1 = \lambda_2 = \lambda_0$, then $N$ is a Poisson process of rate $\lambda_0$, and $V(t) = V(0)(-1)^{N(t)}$, where $V(0)$ has distribution $(1-p,p)$. It follows that
\begin{align}
&\mathsf{P}\{V(t) = v\} = 
\begin{cases}
\frac{1}{2}+(p-\frac{1}{2})e^{-2\lambda_0t}, &v=v_0\\
\frac{1}{2}-(p-\frac{1}{2})e^{-2\lambda_0t}, &v=-v_0
\end{cases}
\end{align}
and so $\mathsf{E}[V(t)] = v_0(2p-1)e^{-2\lambda_0t}$ and $\theta = 0$. The integral of $\sqrt{\lambda}V(\lambda t)$, taking the form
\begin{equation}
X_\lambda(t) = \int_0^{t}\sqrt{\lambda}V(0)(-1)^{N(\lambda s)}ds,
\end{equation}
is a telegraph process with modulus of velocity $\sqrt{\lambda}v_0$, rate of changes of direction $\lambda\lambda_0$ and satisfies $T_\lambda \Rightarrow v_0\lambda_0^{-1/2}W$. Furthermore, the density of $T_\lambda(t)$
%\begin{equation}f_\lambda(x) = \frac{e^{-\lambda\lambda_0 t}}{2\sqrt{\lambda}}\bigg(\lambda\lambda_0 I_0(\lambda\lambda_0\sqrt{\lambda t^2-\lambda x^2})+\frac{\partial}{\partial t}I_0(\lambda\lambda_0\sqrt{\lambda^2t^2-\lambda x^2})\bigg)\end{equation}
is a solution of the telegraph equation
\begin{equation}
\frac{1}{2\lambda\lambda_0}\frac{\partial^2 u}{\partial t^2}+\frac{\partial u}{\partial t} = \frac{v_0^2}{2\lambda_0}\frac{\partial^2u}{\partial x^2}
\end{equation}
and by taking the limit as $\lambda \to +\infty$ the heat equation emerges
\begin{equation}
\frac{\partial u}{\partial t}=\frac{v_0^2}{2\lambda_0}\frac{\partial^2u}{\partial x^2}.
\end{equation}
%Let $\hat{V}$ be a sequence of i.i.d. random variables with support $\mathcal{V} = \{v_1,v_2\}$ and $S$ be a sequence of positive random variables independent of $\hat{V}$, where $\xi_{k}=S_{k}-S_{k-1}$ are independent and exponentially distributed with parameter $\lambda_0$. Then, the semi-Markov process $V$ associated with $(\hat{V},S)$ satisfies\begin{equation*}\mathsf{P}\{V(t) = v\} = \mathsf{P}\{\hat{V}_0 = v\} = p\mathsf{1}_{v_1}(v)+(1-p)\mathsf{1}_{v_2}(v) \;\; t \geq 0\end{equation*}and it follows $\mathsf{E}[V(t)] = \theta = \mathsf{E}[\hat{V}_0] = pv_1+(1-p)v_2$. The integral of $V_\lambda(t) = \mathsf{E}[V(\lambda t)]+\sqrt{\lambda}(V(\lambda t)- \mathsf{E}[V(\lambda t)])$ is\begin{equation}X_{\lambda}(t) = \int_0^t\sqrt{\lambda}\bigl(V(\lambda s)-(pv_1+(1-p)v_2)(1-\sqrt{\lambda}^{-1})\bigr)ds.\end{equation}From Theorem \ref{main Theorem} it follows that\begin{equation}X_\lambda \Rightarrow \alpha + \mu^{-1/2}\gamma W\end{equation}where\begin{align}&\alpha(t) = t\mathsf{E}[\hat{V}_0]&\mu^{-1}\gamma^2 = \mathsf{V}[\hat{V}_0]\frac{\mathsf{E}[\xi_1^2]}{\mathsf{E}[\xi_1]}\end{align}

\end{document}